\documentclass{amsart} 
\usepackage{amsmath}
\usepackage{times}
\usepackage{kpfonts}
\usepackage{amsfonts} 
\usepackage{amsthm}
 \usepackage{amssymb} 
 \usepackage{textcomp} 
 \usepackage{xypic}
 \usepackage{hyperref}
\usepackage{mathrsfs}
\usepackage{color}
\usepackage{microtype}
\usepackage{latexsym}
\usepackage[OT2,T1]{fontenc}
\DeclareSymbolFont{cyrletters}{OT2}{wncyr}{m}{n}
\DeclareFontFamily{OT1}{rsfs}{}

\usepackage{verbatim}
\usepackage{graphicx}
\usepackage{nomencl}
\DeclareFontEncoding{OT2}{}{} 

\newcommand{\Z}{\mathbf{Z}}

\newcommand{\Ext}{\mathrm{Ext}}
\newcommand{\Eis}{\mathrm{Eis}}
\newcommand{\Ad}{\mathrm{Ad}}

 \newcommand{\Hom}{\mathrm{Hom}}

\newcommand{\mcA}{\mathcal{A}}
\newcommand{\SL}{\mathrm{SL}}

\renewcommand{\dim}{\mathrm{dim}}
\newcommand{\Abig}{\mathcal{A}}
\newcommand{\Asmall}{\mathcal{A}_{\lambda-\nil}}

\newcommand{\C}{\mathbf{C}}

\newcommand{\gK}{(\mathfrak{g}, K)}

\newcommand{\R}{\mathbf{R}}

 \DeclareFontShape{OT1}{rsfs}{n}{it}{<-> rsfs10}{}
\DeclareMathAlphabet{\mathscr}{OT1}{rsfs}{n}{it}

\newtheorem{theorem}{Theorem}
\newtheorem{prop}[theorem]{Proposition}
\newtheorem{lemma}[theorem]{Lemma}

\newtheorem{remark}[theorem]{Remark}
\newtheorem{definition}[theorem]{Definition}
\numberwithin{theorem}{section}

\newcommand{\ip}[2]{\langle #1, #2 \rangle}
\newcommand{\gk}{(\mathfrak{g},K)}
\newcommand{\mA}{\mathcal{A}}
\newcommand{\mS}{\mathcal{S}}
\newcommand{\mT}{\mathcal{T}}
\newcommand{\mH}{\mathcal{H}}

\newcommand{\mE}{\mathcal{E}}
\newcommand{\fg}{\mathfrak{g}}
\newcommand{\fs}{\mathfrak{s}}

\newcommand{\Cusp}{\mathrm{Cusp}}
\newcommand{\dom}{\backslash}
\newcommand{\BC}{\mathbf{C}}
\newcommand{\BR}{\mathbf{R}}
\newcommand{\BZ}{\mathbf{Z}}

\newcommand{\BH}{\mathbf{H}}

\newcommand{\Cas}{\mathcal{C}}
\newcommand{\meta}{G}
\newcommand{\del}{\partial}

\newcommand{\nil}{\textrm{nil}}

\begin{document}  
	\title{Cohomology of Fuchsian groups and Fourier interpolation}
	\author{Mathilde Gerbelli-Gauthier and Akshay Venkatesh}
	\begin{abstract} 
	We  give a new proof of a Fourier interpolation result first proved by Radchenko-Viazovska, deriving it from a vanishing result
	of the first cohomology of a Fuchsian group with coefficients in the Weil representation.

	\end{abstract}

	\maketitle

	\setcounter{tocdepth}{1}
	\tableofcontents

	\section{Introduction}\label{introduction}

Let~$\mS$ be the space of even Schwartz functions on the real line, and~$\fs$ the space of sequences of  complex numbers $(a_n)_{n \geq 0}$ such that $|a_n| n^k$ is bounded for all $k$; we write $\hat{\phi}(k) = \int_{\BR} \phi(x) e^{-2 \pi i k x} dx$ for the Fourier transform of $\phi \in \mS$. 
 In \cite{RV19}, Radchenko-Viazovska proved the following
beautiful ``interpolation formula'': 

\begin{theorem} \label{interpolation}
	The map   \[\Psi: \mS\to \fs \oplus \fs, \quad \phi \mapsto (\phi(\sqrt n), \hat{\phi}(\sqrt n))_{n \geq 0} \] 	is an isomorphism onto the codimension 1 subspace of~$\fs \oplus \fs$ cut out by the Poisson summation formula, i.e.  the subspace of $(x_n, y_n)$ defined by
	$\sum_{n \in \mathbf{Z}} x_{n^2}= \sum_{n \in \mathbf{Z}} y_{n^2}$.  
\end{theorem}

	This is an abstract interpolation result: The statement implies the existence of a universal formula
that computes any value~$\phi(x)$ of any  even Schwartz function~$\phi$
as a linear combination $\sum a_n(x) \phi(\sqrt{n}) + \sum \hat{a}_n(x) \hat{\phi}(\sqrt{n})$
for some $a_n(x), \hat{a}_n(x)$, but does not specify what those functions are.  By contrast, 
Radchenko and Viazovska  first write down this explicit interpolation formula, and
then deduce 
 Theorem \ref{interpolation} from it. In a sense, what is accomplished
in the present paper is to separate the abstract content of this interpolation result
from its computational  aspect.
  
The morphism $\Psi$ is in fact a homeomorphism of topological vector spaces with reference to natural topologies. 
We will give another  proof of this theorem. The first step
of this proof is to notice that the evaluation points $\sqrt{n}$
occur very naturally in the theory of the {\em oscillator representation} defined by Segal--Shale--Weil (see \cite{Chan12} or \cite{LV13} for introductions).
Using this observation,  the theorem can be reduced to computing
the cohomology of a  certain Fuchsian group with coefficients in
this oscillator representation, and here we prove a more general statement:

\begin{theorem} \label{mainthm}
 Let $G$ be $\SL_2(\R)$ or a finite cover thereof, $\Gamma$ a lattice in $G$, 
   $W$ an irreducible infinite-dimensional $(\mathfrak{g}, K)$-module, 
 and  $W^*_{-\infty}$  the distributional globalization of its dual (see \ref{ss globalization}). 
 Then
 $H^1(\Gamma, W^*_{-\infty})$ is always finite dimensional, and in fact 
\begin{equation} \label{mainthm FrobRep} \dim H^1(\Gamma, W^*_{-\infty}) =  \mbox{multiplicity of $W^{cl}$ in  cusp forms on }\Gamma\backslash G ,\end{equation}
 where  $W^{cl}$ is the complementary irreducible representation to $W$ defined in  \S \ref{gKmodules}. 
\end{theorem}

The theorem can be contrasted with usual Frobenius reciprocity:
\begin{equation} \label{eq FrobRep} \dim H^0(\Gamma, W^*_{-\infty}) = \mbox{multiplicity of $W$ in  automorphic forms on } \Gamma \dom G, \end{equation}
Note that, in the passage from \eqref{mainthm FrobRep} to \eqref{eq FrobRep}, 
``cusp forms'' have been replaced by ``automorphic forms''
and $W^{cl}$ by $W$. 
We also emphasize the   surprising fact 
that, in the theorem, the $H^1$ takes {\em no account} of the topology on~$W^*_{-\infty}$:
it is simply the  usual cohomology of the discrete group~$\Gamma$ acting on the abstract vector space~$W^*_{-\infty}$. 
The corresponding determination for finite-dimensional $W$ is
the subject of automorphic cohomology and is in particular completely understood, going back to~\cite{Eichler}. 

A variant of Theorem \ref{mainthm},  computing {\em all} the cohomology groups $H^i$ when $W$ is a spherical principal series representation, was already proved by Bunke and Olbrich in the 1990s. 
We were unaware of this work when we first proved Theorem \ref{mainthm}; our original
argument has many points in common with \cite{BO},  most importantly
in our usage of surjectivity of the Laplacian both for analytic and algebraic purposes, 
 but also has some substantial differences of setup and emphasis. 
 We will correspondingly give two proofs: the first based on  the results of \cite{BO}, and the second
a shortened version of our original argument.  

 Some other interpolation consequences
 of Theorem \ref{mainthm}, 
where interpolation is understood
 in the abstract sense as discussed after Theorem \ref{interpolation}, 
  arise by replacing~$\mS$ by other spaces of functions carrying natural representations of $SL_2(\BR)$ and its finite covers:  we discuss this in \S \ref{triangle groups}.
	For example, 
Hedenmalm and Montes-Rodríguez \cite{HM}
have shown that the functions $e^{i \pi \alpha n t}$, $e^{i \pi \beta n/t}$
are weakly dense in $L^{\infty}$ if and only if $\alpha\beta=1$.
We will show that an interpolation result holds at the transition point $\alpha \beta =1$;
we thank the referee for bringing  \cite{HM} to our attention.

\subsection{Theorem \ref{mainthm} implies Theorem \ref{interpolation}}   \label{outline}

We give an outline of the argument and refer to \S \ref{InterpCoh} for details. 
 
We pass first to a dual situation.  Denote by $\mS^*$ the space of tempered distributions, i.e. the continuous dual of $\mS$. 
 For our purposes we regard it as a vector space without
topology.

Similarly, we define $\mathfrak{s}^*$ as the continuous dual of $\mathfrak{s}$, 
where $\mathfrak{s}$ is topologized by means of the norms $\|(b_n)\|_k := \sup_{n} b_n (1+|n|)^k$; 
thus, $\mathfrak{s}^*$  may be identified with sequences $(a_n)$ of complex numbers of polynomial growth,
where the pairing of $(a_n) \in \mathfrak{s}^*$ and $(b_n) \in \mathfrak{s}$ is given by the rule $\sum a_n b _n$.  With this notation, the map  \[\Psi^*: \fs^* \oplus \fs^* \to \mS^*,\] 
dual to $\Psi$
  takes the coordinate functions to the distributions $\delta_{n}$, and $ \hat{\delta}_{n}$,    
  \[ (a_n,b_n)_{n \geq 0}  \mapsto \sum a_n \delta_{n} + b_n \hat{\delta}_{n},  \] where 
  \[ \delta_{n}(\phi) = \phi(\sqrt{n}), \quad \hat{\delta}_{n}(\phi) = \hat{\phi}(\sqrt{n}). \] 

Then Theorem \ref{interpolation} is equivalent to the assertion:
\begin{quote}
(Dual interpolation theorem): 
  $\Psi^*$ is surjective
and its kernel consists precisely of the ``Poisson summation'' relation.
\end{quote}
The equivalence of this statement and Theorem \ref{interpolation} is not a complete formality because of issues of topology: see
\eqref{Frechetdual} for an argument that uses a theorem of Banach. 

 The next key observation is that
 the space of distributions spanned by~$\delta_{n}$ and by~$\hat{\delta}_{n}$
   occur in a natural way in representation theory.  
\begin{quote} 
The closure of the span of~$\delta_n$ (respectively, the closure of the span of $\hat{\delta}_n$)
coincide with the
$e$-fixed and $f$-fixed vectors on the space $\mathcal{S}^*$ of tempered distributions,
where   \begin{equation} \label{efdef} e = \left(\begin{array}{cc}
1&2\\0&1
\end{array}\right)\quad \text{and} \quad f = \left(\begin{array}{cc}
1&0\\2&1
\end{array}\right) \end{equation} 
act on $\mathcal{S}^*$ according  to the oscillator representation (see \S \ref{Weil} for details), namely 
$e$ and $f$ multiply $\phi$ and $\hat \phi$, respectively, by $e^{2 \pi i x^2}$, see \eqref{eq unipotent action}.
\end{quote}

  Let~$\Gamma$ be the group generated by~$e$ and~$f$ inside $\SL_2(\R)$: it is a free group, of index two in $\Gamma(2)$, and it lifts to the double cover $G$ of $\SL_2(\R)$. 
 As explicated in \S \ref{InterpCoh}, computations of dimensions of modular forms and  Theorem \ref{mainthm}  yield
\begin{equation} \label{win} \dim H^0(\Gamma, \mS^*)  =1, \quad \dim H^1(\Gamma, \mS^*) = 0.\end{equation} 

The final observation is that 
\begin{quote} The kernel and cokernel of $(\mS^*)^e \oplus  (\mS^*)^f  \to  \mS^*$
compute, respectively, the $H^0$ and $H^1$ of $\Gamma$ acting on $\mS^*$.
\end{quote}

This follows from a Mayer-Vietoris type long exact sequence 
that computes the cohomology of the free group $\Gamma$ \cite[II-III]{BrownCohomology}, namely, \begin{equation} \label{MWseq}  0 \to  H^0(\Gamma, \mS^*) \to  H^0\left(\langle e \rangle , \mS^*\right) \oplus H^0\left(\langle f \rangle, \mS^*\right) \to H^0(1,\mS^*) \to H^1(\Gamma, \mS^*) \to  \cdots \end{equation} 
Combined with \eqref{win}, we see that  $\mS^* = (\mS^*)^e+(\mS^*)^f$, i.e.  the desired surjectivity of $\Psi^*$,
and that the intersection of $(\mS^*)^e$ and $(\mS^*)^f$ 
is one-dimensional; this corresponds exactly to the Poisson summation formula.

Another way to look at this is the following. 
  The Poisson summation formula
is an obstruction to {\em surjectivity} in Theorem \ref{interpolation} and 
is closely related to the invariance of the distribution $\sum \delta_n \in \mathcal{S}^*$
by $\Gamma$, i.e. the existence of a class in the {\em zeroth} cohomology of $\Gamma$ on~$\mathcal{S}^*$.   The  above discussion
shows a less obvious statement: the obstruction to {\em injectivity}  in Theorem \ref{interpolation} is  precisely the {\em first} cohomology
of~$\Gamma$ on~$\mathcal{S}^*$.

\subsection{The proof of Theorem \ref{mainthm}.}

The analogue of Theorem \ref{mainthm} when $W$ is finite-dimensional
and $\Gamma \backslash G$ is compact   is (by now) a straightforward exercise;
as noted, the ideas go back at least to \cite{Eichler}, and
the general case is documented in   \cite{BW00};
the noncompact case is less standard but also well-known,
see e.g.  \cite{CasselmanHodge} and  \cite{Fr98} for a comprehensive treatment.

The main complication of our case is that the  coefficients are infinite-dimensional and one might think this  renders the question unmanageable.  The key point is that~$W$
is irreducible as a~$G$-module. This says that, ``relative to~$G$'',
it is just as good as a finite-dimensional representation. 

We will present two proofs of Theorem \ref{mainthm}:

\begin{itemize}
\item 
The first proof, in Section \ref{BOsec}, relies on the work of Bunke--Olbrich
who computed the cohomology  of lattices in $\SL_2(\R)$ with coefficients in 
(the distribution globalization of a) principal series representation.  We give a sketch
of the argument of \cite{BO} for the convenience of the reader, and also because
their argument as written does not cover the situation we need. To deduce Theorem \ref{mainthm}
from these results then requires us to pass from a principal series to a subquotient,
which we do in a rather {\em ad hoc} way.

\item 
The second proof is our original argument prior to learning of the work of Bunke-Olbrich just mentioned.
 It generalizes the standard way of computing $\Gamma$-cohomology
with finite dimensional coefficients, as given in \cite{BW00}, to the infinite dimensional case -- at least in cohomological degree $1$. Given
the content of \cite{BO}, we have permitted ourselves to abridge some tedious
parts of our original argument,
and   reproduce here in detail  the part that is perhaps most distinct from \cite{BO} -- 
namely, we  express the desired cohomology groups
in terms of certain $\Ext$-groups of $\gK$-modules and then compute these explicitly. 

\end{itemize}
 In both arguments the surjectivity of a Laplacian type operator
 plays an essential role. Such results are known since the work of Casselman \cite{CasselmanHodge}, and
 in their work,  Bunke--Olbrich prove and utilize such a  result both at the level of $G$ and $\Gamma \backslash G$.  We include a self-contained proof of  such a result for $\Gamma \backslash G$
 in \S  \ref{ourproof}.

 \subsection{Questions}

As we have noted, we prove an {\em abstract} interpolation result.
Can one recover the explicit formula for the interpolating functions, as given in \cite{RV19}, from this approach?
It seems to us that our
 proof is sufficiently explicit that this is, at least, plausible.

 It is very interesting to ask about the situation where $\Gamma$ is not a lattice. 
Indeed, if one were to ask about an interpolation formula
with evaluation points $0.9 \sqrt{n}$, 
one is immediately led to similar questions for a discrete but {\em infinite covolume}
subgroup of $\SL_2(\R)$, whereas considering $1.1\sqrt{n}$ leads to considering a non-discrete subgroup. Note that Kulikov, Nazarov and Sodin \cite{NSK}  have recently shown very general
results about Fourier uniqueness that imply, in particular, that evaluating $f$ and $\hat{f}$
at $1.1 \sqrt{n}$ do not suffice to determine $f$, but that evaluating them at $0.9\sqrt{n}$ \emph{does}.

 Perhaps a more straightforward question is to an establish an isomorphism
 \begin{equation} \label{verygeneral} H^i(\Gamma, W^*_{-\infty}) \simeq \mathrm{Ext}^i_{\mathfrak{g},K}(W,  \mbox{space of automorphic forms for $\Gamma \backslash G$}),\end{equation}
 which is valid for general lattices $\Gamma$ in semisimple Lie groups $G$
 and general irreducible (smooth, moderate growth) representations $V$ of $G$. Bunke and Olbrich have proved this in the cocompact case,
 and our  original argument proceeded by establishing the case $i=1$ for general lattices in $\SL_2(\R)$.
	Also, Deitmar and Hilgert prove \cite[Corollary 3.3]{DeiHil2005} a result of this type in great generality, but with the space of automorphic forms replaced by the larger space $C^{\infty}(\Gamma \backslash G)$
 without growth constraints.

 \subsection{Acknowledgements}
 
We are very grateful to Matthew Emerton for his interest and input on the paper. 
These discussions led us to  the work of Bunke and Olbrich. We also thank Henri Darmon and Joshua Mundinger for useful comments and suggestions. We thank the referees for comments and suggestions that led to improvements in the exposition.
  
 M.G-G. was supported by the CRM and McGill University. A.V. was supported by an NSF-DMS grant during the preparation of this paper.

\section{Covering groups of $\SL_2(\BR)$} \label{section metaplectic group}
 
 Let $q \geq 1$ be a positive integer and let
 $G$ be the $q$-fold covering of the group $\SL_2(\BR)$, i.e.
 $G$ is a connected Lie group  equipped with a  continuous homomorphism
 $G \rightarrow \SL_2(\BR)$ with kernel of order $q$. This characterizes $G$ up to unique isomorphism covering the identity of $\SL_2(\R)$. 
 
   Denote by $\fg$ the shared Lie algebra of $\meta$ and of $SL_2(\BR)$ and $\exp: \fg \to G$ the exponential map.
   Also denote by $K$ the preimage of $\mathrm{SO}_2(\R)$ inside $G$; it is abstractly isomorphic  as topological group to $S^1=\R/\Z$
   and we fix such an isomorphism below. 
   
   The quotient $G/K$ is identified with the hyperbolic plane $\BH$, on which $G$ acts by isometries. Define the norm of $g \in G$ to be $\|g\| := e^{\mathrm{dist}_{\BH}(i,gi)}$. Equivalently, we could use $\left\|\left(\begin{smallmatrix}
   	a&b\\c&d
   \end{smallmatrix}\right)\right\| =\sqrt{a^2+b^2+c^2+d^2}$ since either of these two norms is bounded by a constant multiple of the other.

\subsection{Lie algebra} \label{ss Lie algebra}
Let $H, X, Y$ be the standard basis for $\fg$:
$$ X = \left( \begin{array}{cc} 0 & 1 \\ 0 & 0 \end{array} \right), \quad 
Y = \ \left( \begin{array}{cc} 0 & 0 \\ 1 & 0 \end{array} \right), \quad H= \left( \begin{array}{cc} 1 & 0 \\ 0 & -1 \end{array} \right).$$
 We also use
$ \kappa=i(X-Y)$, $2p = H-i(X+Y)$, $2m = H+i(X+Y)$, or, in matrix form:
\begin{equation} \label{Lie elements def}  \kappa = \left( \begin{array}{cc} 0 & i \\ -i & 0 \end{array} \right), \quad 
2p = \ \left( \begin{array}{cc} 1 & -i \\ -i & -1 \end{array} \right),\quad  2m= \left( \begin{array}{cc} 1 & i \\ i & -1 \end{array} \right).\end{equation}
We have $\kappa = ik$, where $k$ generates the Lie algebra of $K$. 

 The elements $p,m$ and $\kappa$ satisfy the commutation relations 
\begin{equation} \label{commutation} [p, m] = \kappa,\quad  [\kappa, p] =2p,\quad  [\kappa, m]=-2m,\end{equation}
which say that $p$ and $m$ (shorthand for \emph{plus} and \emph{minus}) raise and lower $\kappa$-weights by $2$. 
 The Casimir element $\mathcal{C}$ in the universal enveloping algebra determined by the trace form is given by any of the equivalent formulas:
 \begin{equation} \label{AllTheCasimirExpressions} \mathcal{C}= \frac{H^2}{2} + XY+YX =  \frac{\kappa^2}{2} + pm + mp = \frac{\kappa^2}{2} + \kappa + 2mp = \frac{\kappa^2}{2} - \kappa + 2pm. 
 \end{equation}

\subsection{Iwasawa decomposition} \label{ss Iwasawa}
There is a decomposition \begin{equation}\label{eq Iwasawa}G = NAK\end{equation} where 
 $A$ and $N$ are the connected Lie subgroups of $G$ with Lie algebra $\mathbf{R}. H$ and
$\mathbf{R}. X$ respectively. We will parameterize elements of $A$ 
via
\[a_y := \exp(\frac{1}{2} \log(y) H),\] so that $a_y$ projects
to the diagonal element of $\SL_2(\R)$ with entries $y^{\pm 1/2}$. 
We will also write $n_x = \exp(xX)$.

\subsection{$\gK$-modules} \label{gKmodules}

We recall that a $\gK$ module $W$ means a $\mathfrak{g}$-module
equipped with a compatible continuous action of $K$.
Equivalently, it is described by the following data: 

\begin{itemize}
 
\item For each $\zeta \in q^{-1} \Z$, a vector space $W_{\zeta}$
giving the $\zeta$-weight space of $K$, so that $\kappa$ acts on $W_{\zeta}$ by $\zeta$; 
\item maps $p: W_{\zeta} \rightarrow W_{\zeta+2}$ and $m: W_{\zeta} \rightarrow W_{\zeta-2}$
satisfying $[p, m]=\kappa$. 
\end{itemize}

We recall some facts about classification, see \cite{HT12} for details. Irreducible, infinite-dimensional $(\fg,K)$-modules belong to one of three classes; in each case, the weight spaces $W_\zeta$ have dimension either zero or $1$. 
\begin{itemize}
	\item Highest weight modules of weight~$\zeta$; these are determined up to isomorphism 
	by the fact that their nonzero weight spaces occur in weights $\{\zeta, \zeta-2, \zeta-4, \dots\}$.
	$W_{\zeta}$ is killed by $p$.  One computes using \eqref{AllTheCasimirExpressions} that on such modules, the Casimir element~$\mathcal{C}$ acts by~$\zeta(\zeta+2)/2$. 
 
	\item Lowest weight modules of weight $\zeta$; these are determined up to isomorphism
	by the fact that their nonzero weight spaces occur in weights $\{\zeta, \zeta+2, \zeta+4, \dots\}$.
      $W_{\zeta}$ is killed by $m$. Again, \eqref{AllTheCasimirExpressions} shows that the Casimir element $\mathcal{C}$ acts by $\zeta(\zeta-2)/2$. 
	\item Doubly infinite modules, in which the weights are of the form $\zeta+ 2\BZ$ for $\zeta \in \frac{1}{q}\Z$. 
\end{itemize}
\begin{definition} \label{def complementary}

For an  infinite-dimensional irreducible $\gK$-module $W$ 
we define the complementary irreducible representation $W^{cl}$ to be 
	\[ W^{cl} = \begin{cases}
	\text{the irreducible $(\fg, K)$-module with 
		highest weight $\zeta-2$} & \text{$W$ has lowest weight $\zeta$} \\
	\text{the irreducible $(\fg, K)$-module with 
		lowest weight $\zeta+2$} & \text{$W$ has highest weight $\zeta$} \\
	W & \text{else.}
\end{cases}  \]
\end{definition}
 Note that the representation $W^{cl}$ can be finite-dimensional; this occurs exactly when $W$ is the underlying $(\mathfrak g,K)$-module of a discrete series representation on~$SL_2(\BR)$.

 In \S \ref{section our proof} we use the following key fact about $\gK$-modules.
\begin{prop} \label{comp}
Let $W$ be an irreducible infinite-dimensional $(\fg,K)$-module with Casimir eigenvalue $\lambda$.
Then, for any $\gK$-module $V$:
 \begin{itemize}
\item[(a)] If $\Cas-\lambda$ is surjective on $V$,
then
$\mathrm{Ext}^1_{(\fg,K)}(W, V)= 0$.

\item[(b)]  If $V$ is irreducible,  $\Ext^1_{(\fg, K)}(W,V)$ is one-dimensional if $V \simeq W^{cl}$, and is zero otherwise. 
 
\end{itemize}
 
\end{prop}
\begin{proof}

We will prove these statements in the case where
$W$ is a lowest weight module, which is the case of our main application.
The same proof works with slight modifications for $W$ a highest weight or doubly infinite module: in every case, one takes an arbitrary lift of a generating vector, and modifies it using the surjectivity of an appropriate operator.

We prove (a). Take $W$ to be   generated by a vector $v_{\zeta}$ of lowest weight $\zeta$ with
$m v _{\zeta} = 0$. This implies by the classification above that \begin{equation}\label{eq roots casimir}\lambda = \frac{\zeta(\zeta-2)}{2}.\end{equation} 
Take an extension $V \rightarrow E \rightarrow W$;
to give a splitting we must lift $w_{\zeta}$ to 
a vector in~$E$ of $K$-type $\zeta$
killed by $m$. Arbitrarily lift 
$w_{\zeta}$ to
$\tilde{w}_{\zeta} \in E_{\zeta}$.  
Then $m \tilde{w}_{\zeta} \in V_{\zeta-2}$ and it suffices to show that it lies inside
the image of $m: V_{\zeta} \rightarrow V_{\zeta-2}$, for we then
modify the choice of $\tilde{w}_{\zeta}$ by any preimage to get the desired splitting.  By \eqref{AllTheCasimirExpressions}  and \eqref{eq roots casimir} we see that $\Cas-\lambda: V_{\zeta-2} \rightarrow V_{\zeta-2}$
agrees with $2mp$. Since it is surjective, it follows that
in particular $m: V_{\zeta} \rightarrow V_{\zeta-2}$ is surjective. 
  
We pass to (b).  Suppose $V$ is irreducible;
  then $\mathrm{Ext}^1_{\gK}(W, V)$ vanishes unless $V$ has the same  $\mathcal{C}$-eigenvalue as $W$.  
  The argument above
  exhibits an injection of \[  \mathrm{Ext}^1_{\gK}(W, V) \hookrightarrow  V_{\zeta-2}/mV_{\zeta}\]
and inspection of $K$-types amongst those irreducibles with the same $\mathcal{C}$-eigenvalue as $V$ 
shows that this also vanishes unless $V \simeq W^{cl}$,
in which case it is one-dimensional. It remains only to exhibit a nontrivial extension of $W$ by $W^{cl}$,
which is readily done by explicit computation.
\end{proof}

\subsection{Globalizations} \label{ss globalization}

A globalization of a $\gK$-module $W$ is any continuous $G$-representation  on a topological vector space 
$\overline{W}$  such that $\left(\overline{W}\right)_K=W$. We will consider two instances of this: the \emph{smooth, or Casselman-Wallach globalization} $W_\infty$, and the \emph{distributional globalization} $W_{-\infty}$. 

Following \cite{Ca89}, the representation $W_\infty$ is the unique globalization of $W$ as a moderate growth Fr\'echet $G$-representation. By definition, such a representation is a Fr\'echet space $F$ (topologized with respect to a family of seminorms) such that for any seminorm $\|\cdot \|_\alpha$, there is an integer $N_\alpha$ and a seminorm $\| \cdot \|_\beta$ for which \[ \|gw\|_\alpha \leq \|g\|^{N_\alpha} \|w\|_\beta. \]

The distributional globalization is a dual notion. Indeed, denote by $W^*$ the $K$-finite part of the dual of $W$, equipped with the contragredient~$\gk$-module structure. Then \begin{equation}
	\label{eq dualitycompletions} (W_\infty)^* = (W^*)_{-\infty}
\end{equation} where on the left-hand side, the dual is understood as continuous.

We recall an explicit 
construction of $W_{-\infty}$, see \cite[\S 2-3]{BO}, although it will not be directly used in the rest of the paper:  Given $W^*$ as above, let $V^*\subset W^*$ be a finite-dimensional~$K$-stable subspace that generates $W^*$ as a $(\fg,K)$-module. Let $(V^*)^*=:V\subset W$ viewed as a   $K$-representation, and consider the space
\[ \mE_{V} = \left\{f \in C^\infty(G,V) \mid f(gk)=k^{-1}f(g),\,g \in G,\, k \in K \right\}. \] Then the image of $W$ under the map $i: W \to \mE_{V}$ characterized by 
$ \ip{ i(w)(g)}{v^*} := \ip{w}{gv^*} \ \ (w \in W, \; v^* \in V^*) $ belongs to the space $\mA^G_{V}$ of sections of moderate growth, i.e. of functions $f \in \mE_{V}$ such that for every  $X \in U(\fg)$, there is $R=R(f,X)$ for which \begin{equation} \label{moderate growth} \|f\|_{X,R} = \sup_{g\in G} \frac{|Xf(g)|}{\|g\|^R}<\infty. \end{equation}
We note that this differs from the notion of {\em uniform} moderate growth, where one requires $R$ to be taken independently of $X$. 

The space $\mA_V^G$ is topologized as the direct limit of Fr\'echet spaces with respect to the seminorms $\|\cdot \|_{X,R}$. The map $i$ is injective since $V^*$ generates $W^*$, and the distributional globalization is defined as \[W_{-\infty}:=\overline{i(W)} \subset \mA_V^G.\]

\newcommand{\sigmaximacro}{}
\newcommand{\sigmaximacrocomma}{}

 \section{First proof of Theorem \ref{mainthm}: resolutions of principal series.}  
\label{BOsec} 
In this section, we derive Theorem \ref{mainthm} from the results of Bunke--Olbrich \cite{BO}, adapting the arguments of Section 9 therein to non-spherical principal series. 
The two essential ingredients of this argument are the following points established by Bunke--Olbrich, 
which we shall use as ``black boxes'': 
\begin{itemize}
\item acyclicity of $\Gamma$ acting on spaces of moderate growth functions on $G/K$,  and
\item surjectivity of a Laplace-type operator acting on these spaces.
\end{itemize}
 The first point, at least, is intuitively reasonable: it asserts that moderate growth functions on $G/K$
behave like a free $\Gamma$-module; this is plausible since the $\Gamma$-action on $G/K$ is (at least, virtually) free. 

Given these, the idea of the argument for Theorem \ref{mainthm} is as follows. We will first show that principal series representations
are realized as spaces of moderate growth Laplacian eigenfunctions
on $G/K$; by the two points mentioned above, this   gives a resolution of the principal series by $\Gamma$-acyclic modules.
This permits us to compute cohomology of principal series representations. Finally, every irreducible representation
is realized as a subquotient of such a representation, and we will then prove Theorem \ref{mainthm} by a study
of the associated long exact sequence in cohomology.

\subsection{Setup}

Fix a Casimir eigenvalue $\lambda$, and a lattice $\Gamma \subset G$. Given $\zeta$ a one-dimensional representation of $K$, define the following spaces of smooth functions (compare with \ref{ss globalization}, and see \eqref{moderate growth} in particular for the notion of moderate growth,
which is {\em not} the  same as uniform moderate growth):  
\begin{align} \label{eq function}
	\mA^G, \;( \text{resp. } \mA ) &= \text{moderate growth functions on }G\; (\text{resp. on }\Gamma \dom G). \\  \nonumber
	\mA_\zeta^G,\;\mA_\zeta&=\text{subspace with right $K$-type $\zeta$: \ \ $f(gk)=f(g)\zeta(k)$.} \\ \nonumber
	\mA_\zeta^G(\lambda),\; \mA_\zeta(\lambda)&=\text{subspace with right $K$-type $\zeta$ and Casimir eigenvalue $\lambda$.}\\
	\Cusp_\zeta(\lambda)&=\text{ subspace of }\mA_\zeta(\lambda)\text{ consisting of cuspforms}. \nonumber
\end{align}

We will first prove a variant of Theorem \ref{mainthm} for principal series.  Let $B$ be the preimage of the upper triangular matrices inside $G$, which we recall is thet $q$-fold cover of $SL_2(\BR)$; 
we may write  
 \[B = MAN\]  where $A$ and $N$ are as in \eqref{eq Iwasawa}, and $M=Z_K(A) \simeq \BZ/2q\BZ$. Denote by $\xi \in \BC$ the character of $A$ sending $a(y) \mapsto y^\xi$. Given a pair of characters $(\sigma,\xi)$ of $K$ and $A$ respectively, let 
\begin{equation} \label{Hsigmaxidef}  H^{\sigmaximacro} = \{ f \in C^\infty(G) \mid f(mang) = a^{\xi+1}\sigma^{-1}(m)f(g), f \text{ $K$-finite} \} \end{equation} be the Harish-Chandra module of $K$-finite vectors in the corresponding principal series representation. 
This depends on $\sigma$ and $\xi$, but to simplify the notation we will not include them explicitly. 
We denote by
$H^{\sigmaximacro}_{-\infty}$ its distributional completion (\S \ref{ss globalization}); 
explicitly,  if we identify $H^{\sigmaximacro}$ as above
with functions on $K$ which transform on the left under the character $\sigma^{-1}$, 
then $H^{\sigmaximacro}_{-\infty}$ is the corresponding space of {\em distributions} on $K$.

 Let us explicate this in the language of
\S \ref{gKmodules}. We will parameterize $\sigma$ by the value of $d\sigma$ 
at $\kappa$; this is a class in $q^{-1} \Z$ that we will denote by $\zeta_0$. 
A $K$-basis of $H$ is given by vectors  $e_{\zeta}$ with $\zeta \in \zeta_0 + 2\Z$,
normalized to take value $1$ at the identity of $G$.  
The actions of raising and lowering operators
are as follows: 
\begin{equation} \label{Raising lowering H} p e_{\zeta} =  \frac{1}{2} (\zeta+1+\xi) e_{\zeta+2}  \mbox{ and } m e_{\zeta} = \frac{1}{2} (-\zeta+1+\xi) e_{\zeta-2},\end{equation} 

and the action of the Casimir on $e_{\zeta}$is thereby given by
$ \frac{\xi^2-1}{2}$. 
From these explicit formulas we readily deduce the following statements:
 \begin{itemize}
\item[(a)] If $1+\xi$ does not belong to $\pm \zeta_0+2\Z$, then $H$ is irreducible.
\item[(b)]  If $1+\xi$ belongs to either $ \zeta_0 +2\Z$ or $-\zeta_0+2\Z$ but not both, 
then $H$ 
has the structrure
\begin{equation} \label{eq decomposable principal series} 0 \to \bar{V}^{\sigmaximacro} \to H^{\sigmaximacro} \to  V^{\sigmaximacro} \to 0. 
\end{equation} where  $\bar{V}, V$ are irreducible $\gK$-modules; $\bar{V}$ is the module of highest (lowest)
weight $\zeta$ according to whether $-\xi-1$ or $1+\xi$ belongs to $\zeta_0+2\Z$, and $V= \bar{V}^{cl}$.

\item[(c)] If $1+\xi$ belongs to both\footnote{This happens only when $\zeta_0 \in \Z$, and in particular
the representation descends to a representation of $\SL_2(\R)$.} $\zeta_0+2\Z$  and $-\zeta_0+2\Z$,
and $\xi \geq 1$, then
$H$ has the structure of an extension
$$   V^+ \oplus V^- \rightarrow H \rightarrow F,$$
where  $V^-$ is the highest weight representation of weight $-\xi-1$, and $V^+$ the lowest weight representation of weight $\xi+1$,
whereas $F$ is the finite-dimensional representation of dimension $\xi$
with weights $-\xi+1, -\xi+3, \dots, \xi-1$.  A similar dual description is valid when $\xi \leq 0$,
where the finite dimensional representation now occurs as a subrepresentation.
 
\end{itemize}

   In the following proposition, we will assume that
   we are in either cases (a) or (b) of the above classification -- that is to say, 
    $H^{\sigmaximacro}$ is either irreducible, or decomposes as  
\begin{equation} \label{eq decomp PS} 0 \to \bar{V}^{\sigmaximacro} \to H^{\sigmaximacro} \to  V^{\sigmaximacro} \to 0. 
\end{equation} where both the subrepresentation and quotient are irreducible $\gK$-modules. 

\begin{prop}\label{prop nonspherical OB}
	Let $G$ be the degree $q$ connected cover of $SL_2(\BR)$. Denote by $\lambda$ the eigenvalue by which $\mathcal{C}$ acts on $H^{\sigmaximacro}_{K}$; then
	there are natural isomorphisms 
	\begin{align*}
		H^0(\Gamma, H^{\sigmaximacro}_{-\infty}) &\simeq  \mA_\zeta(\lambda) \\ H^1(\Gamma, H^{\sigmaximacro}_{-\infty}) &\simeq \Cusp_\zeta(\lambda) \\
		H^i(\Gamma, H^{\sigmaximacro}_{-\infty}) &= 0 \qquad\qquad  \text{for }i\geq 2. 
	\end{align*}
where $\zeta$ is any $K$-weight generating the dual $\gK$-module $H^{\sigmaximacrocomma*}$.
\end{prop}

 The condition on $\zeta$  is automatic when $H$ is irreducible, and 
in the case when $H$ is reducible is equivalent to asking that $\zeta$ belongs to the $K$-weights of $\bar{V}^{\sigmaximacrocomma*}$.

\begin{proof}  

In \S 9 of \cite{BO} this result is proven in the case of $q=1$
and the trivial $K$-type. 
We will outline the argument to make clear that it remains valid in the 
situation where we now work, i.e., permitting a covering of $\SL_2(\R)$ and an arbitrary $K$-type. 
 
Fix $v^* \in H^{\sigmaximacrocomma*}$ of $K$-type $\zeta$, normalized
as in the discussion preceding
\eqref{Raising lowering H}.
 Then the rule sending $\mathcal{D} \in H^{\sigmaximacro}_{-\infty}$
to the function $\mathcal{D}( g v^*)$ on $G$ induces
an {\em isomorphism} 
 \begin{equation} \label{eq Poisson transform} H^{\sigmaximacro}_{-\infty} \simeq \mA^G_\zeta(\lambda). 
	\end{equation} 
We will outline a direct proof of this isomorphism. Injectivity, at least, follows readily: if $\mathcal{D}$ lies in the kernel, it would 
annihilate the $\gK$-module generated by $v^*$, which is all of $H^{\sigmaximacrocomma*}$, and by continuity $\mathcal{D}$ is then zero. 

For surjectivity, one first checks that $K$-finite functions lie in the image of the map -- that is to say,  a function $f$ of fixed right- and left- $K$-type, and with a specified
 Casimir eigenvalue, occurs in the image of the map above.  Such an $f$ is uniquely specified up to constants: using the decomposition $G=KAK$,
 the Casimir eigenvalue amounts to a second order differential equation for the function
 $y \mapsto f(a_y)$ for $y \in (1, \infty)$, and of the two-dimensional space of solutions 
 only a one-dimensional subspace extends smoothly over $y=1$; see
 \cite[p. 12-13]{Kitaev} 
 for an explicit description both of the differential equation and a hypergeometric
 basis for the solutions.\footnote{There are other references in the mathematical literature but Kitaev
 explicitly considers the universal cover.}  It follows from this uniqueness that $f$ must agree with  $\mathcal{D}(gv^*)$
 where $\mathcal{D}$ and $v^*$ match the left- and right- $K$-types of $f$. 
 To pass from surjectivity onto $K$-finite vectors to surjectivity, we take arbitrary $f \in \mA^G_{\zeta}(\lambda)$
 and expand it as a sum $\sum_{\xi} f_{\xi}$ of left $K$-types.  Each $f_{\xi}$ has a preimage $v_{\xi}$ according to the previous argument;
 so one must verify that $\sum_{\xi} v_{\xi}$ converges inside $H^{\sigmaximacro}_{-\infty}$, and for this it is enough to show that $\|v_{\xi}\|$ grows polynomially with respect to $|\xi|$ (here we compute $\|v_{\xi}\|$ as the $L^2$-norm restricted to $K$ in \eqref{Hsigmaxidef}).  For this  we ``effectivize'' the previous argument: The moderate growth property of $f$ implies 
a bound of the form $|f_{\xi}(g)| \leq c \|g\|^N$, uniform in $\xi$.
 On the other hand, $f_{\xi} = v_{\xi} (g v^*)$, and such a matrix coefficient always is not too small:
 \begin{equation} \label{mcbound} |v_{\xi} (g v^*)| \geq (1+|\xi|)^{-M} \|v_{\xi}\|
 \mbox{ for  {\em some} choice of }  \|g\| \leq (1+|\xi|)^M.
 \end{equation}
Such lower estimates on matrix coefficients can be obtained by keeping track of error bounds
in asymptotic expressions. They are developed in greater generality 
 in the Casselman--Wallach theory, see  e.g. Corollary 12.4 of \cite{BK14}
 for a closely related result.  Combining \eqref{mcbound} with the upper bound on $f_{\xi}$ shows that
  $\|v_{\xi}\| \leq  c (1+|\xi|)^{MN+M}$ as desired.

This concludes our sketch of proof of
\eqref{eq Poisson transform}, that is to say,  $H^{\sigmaximacro}_{-\infty}$
	 is the kernel of 
	\begin{equation} \label{BO sketch} \mA^G_\zeta \stackrel{\Cas - \lambda}{\longrightarrow} \mA^G_\zeta,\end{equation} 
in the notation of \eqref{eq function}. 
We now invoke surjectivity of a Laplace operator: the morphism $\Cas-\lambda$ of \eqref{BO sketch} is surjective, by \cite[Theorem 2.1]{BO}; 
and consequently \eqref{BO sketch} is in fact a resolution of $H^{\sigmaximacro}_{-\infty}$. 

 Moreover,  \cite[Theorem 5.6]{BO} asserts that the higher cohomology of $\Gamma$ acting on $\mathcal{A}^G_{\zeta}$ 
 vanishes; it is for this argument that Bunke and Olbrich use ``moderate growth'' rather than ``uniform moderate growth.'' 
	 Consequently, the $\Gamma$-cohomology of $H^{\sigmaximacro}_{-\infty}$ 
	 can be computed by taking $\Gamma$-invariants on the complex \eqref{BO sketch}:
	 	\begin{equation*} (\mA^G_\zeta)^{\Gamma} \stackrel{\Cas - \lambda}{\longrightarrow} (\mA^G_\zeta)^{\Gamma}. \end{equation*} 
	 Clearly, the $H^0$ here coincides with $\mA_\zeta(\lambda)$.
	 On the other hand, the image of $\Cas-\lambda$ contains the orthogonal complement of cusp forms
	 (see \cite[Thm. 6.3]{BO}, cf. Proposition \ref{casprop}), and so the $H^1$ coincides 
	 with the cokernel of $\Cas-\lambda$ acting on cusp forms; there we can pass to the 
	 orthogonal complement and identify 
	 $H^1 \simeq \Cusp_\zeta(\lambda)$ as desired.\footnote{In fact, $\Cas-\lambda$ is adjoint to $\Cas-\bar{\lambda}$, 
	 but the kernel of the latter of either is only nonzero if $\lambda$ is real, so we do not keep track of the complex conjugate.}
	 \end{proof}

	The following lemmas will be useful in the sequel.  We omit the proof of the first one.
	\begin{lemma} \label{missing isomorphism}
	Let $\zeta$ be, as in Proposition \ref{prop nonspherical OB}, a $K$-weight on $H^{\sigmaximacrocomma *}$
	which generates the latter as $\gK$-module; fix $v_{\zeta} \in H^*$ nonzero of weight $\zeta$.   For any $\gK$-module $V$,
	there is an isomorphism
	\begin{equation} \label{missing lemma eqn} \Hom_{\gK}(H^{\sigmaximacrocomma *}, V) \rightarrow V_{\zeta}(\lambda), \quad  f \mapsto f(v_{\zeta}),\end{equation}
	where $V_{\zeta}(\lambda)$ is the subspace of $V_{\zeta}$ killed by $\Cas-\lambda$. 
	\end{lemma}

	The second is a precise statement of Frobenius reciprocity, stated in a less formal way in
 \eqref{eq FrobRep}. 
 	\begin{lemma} \label{FrobRep}
	Let $V$ be a finite length $\gK$-module. 
	Then there is an isomorphism
	$$ H^0(\Gamma, V^*_{-\infty})  \simeq \Hom_{\gK}(V, \mA_K),$$
	where $V^*_{-\infty}$ is the distributional globalization of $V^*$. 
	\end{lemma}
 
One of the earliest versions of such a statement can be found in  \cite[Chapter 1, \S 4]{GGP}. For completeness we outline the proof, in our language, in Remark \ref{FrobRepProof}.

 For reducible principal series as in \eqref{eq decomp PS}, we prove:
\begin{prop} \label{propquotient} 
	Let $H^{\sigmaximacro}_{-\infty}$ with  Casimir eigenvalue $\lambda$ decompose as in \eqref{eq decomp PS}. Then
	the quotient map $H \rightarrow V$ induces an isomorphism, after passing to distribution globalizations and $\Gamma$-cohomology,
	 \[ H^1(\Gamma, V^{\sigmaximacro}_{-\infty})\simeq H^1(\Gamma, H^{\sigmaximacro}_{-\infty}) \left( \simeq \Cusp_\zeta(\lambda), \mbox{
	 by Proposition \ref{prop nonspherical OB}} \right). \] 
 \end{prop}

\begin{proof}
 The discussion around \eqref{eq decomposable principal series} shows that
inverting both $\xi$ and $\sigma$ gives rise to another principal series $\bar{H}$ which fits into the exact sequence
 \begin{equation} \label{eq decomp comp ps} 0 \to {V}^{\sigmaximacro} \to \bar{H}^{\sigmaximacro} \to  \bar{V}^{\sigmaximacro} \to 0,  \end{equation} i.e. for which the roles of subrepresentation and quotient are swapped between $\bar{V}^{\sigmaximacro}$ and $V^{\sigmaximacro}$. 
We will deduce the result  by playing off Proposition \ref{prop nonspherical OB}  applied to  (the distribution globalization of) $H$, and the same Proposition applied to $\bar{H}$.

We first consider the long exact sequence associated to (the distribution globalization of) \eqref{eq decomp PS}, namely: 
\begin{align} \label{eq long exact sequence}  0 &\to H^0(\Gamma,\bar V^{\sigmaximacro}_{-\infty}) \to H^0(\Gamma, H^{\sigmaximacro}_{-\infty})   \stackrel{\Omega}{\to}H^0(\Gamma, V^{\sigmaximacro}_{-\infty}) \to \\ \nonumber&\to H^1(\Gamma,\bar V^{\sigmaximacro}_{-\infty}) \to H^1(\Gamma, H^{\sigmaximacro}_{-\infty}) \stackrel{\Pi}{\to}  H^1(\Gamma, V^{\sigmaximacro}_{-\infty}) \to 0. \end{align}  
We have used here that the next group $H^2(\Gamma,\bar V^{\sigmaximacro}_{-\infty})$ 
of the sequence vanishes:
it is isomorphic to $ H^3(\Gamma,V^{\sigmaximacro}_{-\infty})$
by  the long exact sequence associated to \eqref{eq decomp comp ps}
and Proposition \ref{prop nonspherical OB}, and that $H^3$ vanishes always.
Indeed, let $\overline{\Gamma}$ be the image of $\Gamma \rightarrow \mathrm{PSL}_2(\R)$,
and $\mu \leqslant \Gamma$ the kernel of $\Gamma \rightarrow \overline{\Gamma}$;
if $V$ is a $\C[\Gamma]$-module then $H^i(\Gamma,V) = H^i(\overline{\Gamma}, V^{\mu})$,
and, being a lattice in $\mathrm{PSL}_2(\R)$, the virtual cohomological dimension of $\overline{\Gamma}$ is at most $2$.

We  must show that the penultimate map $\Pi$ of \eqref{eq long exact sequence} is an isomorphism.  For this
it is enough to show that $$ \dim \ \mathrm{cokernel} \ \Omega  \geq \dim \ H^1(\Gamma,\bar V^{\sigmaximacro}_{-\infty}).$$

By applying Proposition \ref{prop nonspherical OB} to $\bar{H}^{\sigmaximacro}$, we find that $H^1(\Gamma, \bar{V}^{\sigmaximacro}_{-\infty})$ is a quotient of $\Cusp_{\chi}(\lambda)$, for $\chi$ a weight in 
$V^{\sigmaximacrocomma *}$. 
It therefore suffices to show that
\begin{equation} \label{coker dim}  \dim  \ \ \mathrm{cokernel} \ \Omega \geq \dim\ \Cusp_{\chi}(\lambda).
\end{equation} 
 
We will prove this by exhibiting a subspace
\begin{equation} \label{cuspsub} H^0(\Gamma, V_{-\infty})^{\mathrm{cusp}} \subset H^0(\Gamma, V^{\sigmaximacro}_{-\infty})\end{equation}  of the codomain
 of $\Omega$,  which does not meet the image of $\Omega$, and whose dimension equals that of
  $\Cusp_{\chi}(\lambda)$.

The space $H^0(\Gamma, V^{\sigmaximacro}_{-\infty})$  
 is identified, by means of Frobenius reciprocity   (Lemma \ref{FrobRep})
with the space of homomorphisms  from the dual $\gK$-module  
$V^{\sigmaximacrocomma *}$ to the $K$-finite vectors $\mA_K$ in the space of automorphic forms. We define  $H^0(\Gamma, V_{-\infty})^{\mathrm{cusp}}$ 
 to be the subspace corresponding
 to homomorphisms $V^* \rightarrow \mA_K$ that are actually valued in cusp forms. 
 We now show the two properties of this subspace  $H^0(\Gamma, V_{-\infty})^{\mathrm{cusp}}$
 asserted after \eqref{cuspsub}:
 
 \begin{itemize}
\item[-] Its dimension  equals that of $\Cusp_{\chi}(\lambda)$. 
To see this, apply Lemma \ref{missing isomorphism} to $\bar{H}$, with $\zeta=\chi$ and
$V$ the $K$-finite vectors of the space of cusp forms; it   yields an isomorphism
$$  \Hom_{\gK}(H^{\sigmaximacrocomma *}, \Cusp_{K}) \simeq \Cusp_{\chi}(\lambda).$$
 But homomorphisms from $H^*$
to $\Cusp_K$ factor through $V^{\sigmaximacro *}$
by semisimplicity of the space of cusp forms (which in turn follows by unitarity).  
This shows that the space $\Hom_{\gK}(V^*, \Cusp_K)$ has the same dimension as $\Cusp_{\chi}(\lambda)$, as required.

\item[-] It intersects trivially the image of~$\Omega$.
This amounts to the statement that no homomorphism from $V^{\sigmaximacrocomma *}$
to $\Cusp_K$ can be extended to a homomorphism from $H^{\sigmaximacrocomma *}$
to $\mA_K$. 
Suppose, then, that  $f: H^{\sigmaximacrocomma*} \rightarrow \mA_K$
is a $\gK$-module homomorphism 
whose restriction to $V^{\sigmaximacrocomma *}$ is nonzero and has cuspidal image. 
We now make use of the
 orthogonal projection map from all automorphic forms
 to cusp forms, 
 which exists because  one can sensibly take
the inner product of a cusp form with any function of moderate growth. 
Post-composing $f$ with this projection gives a morphism
from $H^{\sigmaximacrocomma*}$ to the semisimple $\gK$-module $\Cusp_K$;
 since $H^{\sigmaximacrocomma*}$ is a nontrivial extension of $\bar{V}^{\sigmaximacrocomma *}$ by $V^{\sigmaximacrocomma *}$, this morphism is necessarily
trivial on the subrepresentation $V^{\sigmaximacrocomma *}$, a contradiction. 
\end{itemize}
 
\end{proof}

	 Now let us deduce Theorem \ref{mainthm}.   We divide into three cases 
	 according to how the representation $W$ of the theorem can be fit into a principal series. Our division
	 corresponds  
	 to the division (a), (b), (c) enunciated after
	 \eqref{Raising lowering H}, and the statements below about the structure of $W$ can all be deduced from the statements given there.
	 \begin{itemize}
	 \item $W$  is an irreducible principal series, equivalently, $W$ is doubly-infinite. In  this case, $W^{cl} = W$, and combining Proposition \ref{prop nonspherical OB} and Lemma \ref{missing isomorphism}
	 gives the statement of Theorem \ref{mainthm}. 
	 
	  \item $W$ is an irreducible subquotient of a principal series $H$ with exactly two composition factors. 
	  In this case we can suppose that  $W=V^{\sigmaximacrocomma*}$
with notation as in \eqref{eq decomp PS}.   In that notation  we have $W^* = V^{\sigmaximacro}$, and $W^{cl} = \bar{V}^{\sigmaximacrocomma*}$.  Proposition \ref{propquotient} gives $H^1(\Gamma, V^{\sigmaximacro}_{-\infty}) \simeq \Cusp_{\zeta}(\lambda)$, and  Lemma \ref{missing isomorphism}
shows that $\Cusp_{\zeta}(\lambda)$ is identified with the space of $\gK$-homomorphisms
from $H^{\sigmaximacrocomma *}$ to the space of cusp forms;  by semisimplicity of the target
such a homomorphism factors through  the irreducible quotient $\bar{V}^{\sigmaximacrocomma*}=W^{cl}$.
This proves Theorem \ref{mainthm} in this case. 
	  
	  \item $W$ is an irreducible subquotient of a principal series with more than two composition factors.
	  In this case, $W$ is necessarily   a highest- or lowest- weight module factoring through $\SL_2(\R)$,
	  and there is an exact sequence 
	  \begin{equation} \label{fhd} F \rightarrow H \rightarrow \mathcal{D},\end{equation} 
	  where $F$ is finite-dimensional and $\mathcal{D}$ is the sum of $W^*$ and another highest- or lowest- weight module. 
	  Here, $W^{cl} = F^* \simeq F$ 
	  and Theorem \ref{mainthm} is equivalent to the vanishing of $H^1(\Gamma, W^*_{-\infty})$. 
  In the case of a discrete series that factors through $\mathrm{PSL}_2(\R)$, 
this vanishing follows from  \cite[Prop. 8.2]{BO}, 
and 
  the remaining case of an ``odd'' discrete series is handled by the same argument. Namely, use the
  long exact sequence associated to
  \eqref{fhd};
the argument of Proposition \ref{prop nonspherical OB}  shows that $H^1(\Gamma, H_{-\infty})=0$, and also $H^2(\Gamma, F) =0$  by Poincar{\'e} duality because $F$ is nontrivial. Thus also $H^1(\Gamma, \mathcal{D}_{-\infty})=0$ and so its summand $H^1(\Gamma, W^*_{-\infty})$ also vanishes.

	  \end{itemize}

  \section{Second proof of Theorem \ref{mainthm}:   extensions of $\gK$-modules} \label{section our proof}

 Our original proof of Theorem \ref{mainthm} proceeds by a reduction to a computation in the category of $\gK$ modules. 
 The two essential ingredients of this argument are the following points:
\begin{itemize}
\item[(a)] the Casselman--Wallach theory \cite{Ca89,Wa92} which  gives a canonical equivalence between suitable categories
of {\em topological} $G$-representations and {\em algebraic} $\gK$-modules. 
\item[(b)] surjectivity of a Laplace-type operator acting, now, on spaces of moderate growth functions on $\Gamma \backslash G$;
\end{itemize}
We will not prove (a), although we will briefly sketch an  elementary proof
of what we use from it. We will prove (b) in the next section. 
 
 Let $\lambda$ be the eigenvalue by which the Casimir 
 $\Cas \in Z(\fg)$ of
 \eqref{AllTheCasimirExpressions} acts on
 $W$ (the irreducible $\gK$-module from the statement of Theorem \ref{mainthm}). We will use the notation $\mA$ from \eqref{eq function} for the space of smooth, {\em uniform} moderate growth functions~$f$ 
 on $\Gamma \dom G$,
 i.e. for which there exists $R$ such that for all $X \in \mathfrak{U}$,  \begin{equation} \label{uniform moderate growth} \|f\|_{X,R} = \sup_{g\in G} \frac{|Xf(g)|}{\|g\|^R}<\infty. \end{equation} 
(compare with \eqref{moderate growth}, and beware that we are using the same notation as in Section~\ref{BOsec}, but for a slightly different space).  We use {\em uniform} moderate growth because
it interfaces more readily with the Casselman--Wallach theory;
by contrast, \S \ref{BOsec} used moderate growth because this is used in the acyclicity result
mentioned after \eqref{BO sketch}. 
 
 Also consider the following subspaces of $\mA$: \begin{align*}
 	\mA_{\lambda-\nil} &= \text{ $K$-finite functions on which $\Cas-\lambda$ acts nilpotently},\\
 	\Cusp(\lambda) &= \text{ subspace of $\mA_{\lambda-\nil}$ consisting of cusp forms.}
 \end{align*}

The precise form of (b) we will use is this:

\begin{prop} \label{casprop}
The image of the map~$\Cas-\lambda: \mA_{K} \rightarrow \mA_K$
is precisely the orthogonal complement to $\Cusp(\lambda)$ inside $\mA_K$. 
  \end{prop}
  
  This is almost \cite[Theorem 6.3]{BO}, except there
  the argument is for moderate growth functions rather than uniform moderate growth;
  they state  on {\em op. cit.} p. 73 that the same proof remains valid in the uniform moderate growth setting.
Also, Cassleman proves \cite[Theorem 4.4]{CasselmanHodge}, for the trivial $K$-type, that $\Cas$ 
is surjective on spaces of Eisenstein distributions, from which a similar result can be extracted.
Since the above statement is in a sense the crux of the argument, 
and neither reference gives it in precisely this form,    
  we have given a self-contained proof in \S \ref{ourproof}.   Our
  proof follows a slightly different strategy and is perhaps of independent interest.

\subsection{Proof of Theorem \ref{mainthm}: reduction to $\gK$ extensions}
We begin the proof of Theorem \ref{mainthm}
 assuming Proposition \ref{casprop}.   
  This will proceed in three steps: 
\begin{itemize}
	\item[(i)] First, using a topological versions of Shapiro's lemma, we identify $H^1(\Gamma, W^*_{-\infty}) \simeq \Ext_G^1(W_\infty, \mA).$
	\item[(ii)] Next, we pass from the category of $G$-modules to that of $(\mathfrak g,K)$-modules and produce an  isomorphism $\Ext_G^1(W_\infty, \mA)\simeq \Ext_{\gK}^1(W, \mA_{\lambda-nil})$. 
	\item[(iii)] Finally, we compute that $\Ext_{\gK}^1(W, \mA_{\lambda-nil})$ is isomorphic to the promised space of cuspforms, using the explicit  computations from \S \ref{gKmodules}.
\end{itemize}

In practice, for technical reasons, we carry out (iii)  first and then show that the map of (ii) is an isomorphism.

We begin by constructing an isomorphism  
\begin{equation} \label{orig} H^1(\Gamma, W^*_{-\infty}) \simeq \Ext_G^1(W_{\infty}, \Abig), \end{equation}
where $W_{\infty}$ is the smooth globalization of $W$.  

On the left, we have the ordinary group cohomology of the discrete group $\Gamma$ acting
on the vector space $W^*_{-\infty}$, without reference to topology.
On the right here we use a {\em topological} version of~$\Ext$
defined as follows: present $\Abig$ as a directed union $\varinjlim \Abig(R)$
of moderate growth Fr{\'e}chet~$G$-representations (see 
\ref{ss globalization})
 obtained by imposing a specific
exponent of growth~$R$ in \eqref{uniform moderate growth}.   The right hand side
is then defined to be the direct limit $\varinjlim \Ext^1_G(W_{\infty},  \Abig(R))$, where
the elements of each~$\Ext$ group are   represented by isomorphism classes of 
short exact sequences\footnote{Here, the notion
  of exact sequence is the usual one, with no reference to topology: the first map is injective, and its image is the kernel of the second, surjective map.}  $\Abig(R) \rightarrow ? \rightarrow W_{\infty}$,
with~$?$ a moderate growth Fr{\'e}chet $G$-representation
and the maps are required to be continuous.

  The statement \eqref{orig} is  then a version of Shapiro's lemma in group cohomology.  Let us spell out
  the relationship:  for $G_1 \leq G_2$ of finite index,
and $W$ a finite-dimensional $G_1$-representation, 
Shapiro's lemma supplies
 an  isomorphism  
\begin{equation} \label{orig.shapiro} H^1(G_1, W^*) \stackrel{(i)}{\simeq} H^1(G_2, \mathrm{I}^{G_2}_{G_1} W^*)  \stackrel{(ii)}{\simeq} 
 H^1(G_2, ( \mathrm{I}^{G_2}_{G_1}  \C) \otimes W^*)  \stackrel{(iii)}{\simeq} \mathrm{Ext}^1_{G_2}(W, \mathrm{I}^{G_2}_{G_1} \C).\end{equation} 
Here  $I^{G_2}_{G_1}$ is the induction from $G_1$ to $G_2$, and we used in (i) Shapiro's lemma in its standard form \cite[III.5-6]{BrownCohomology}; at step (ii) the projection formula
$\mathrm{I}^{G_2}_{G_1} W^* \simeq \mathrm{I}^{G_2}_{G_1} \C \otimes W^*$,
and at step (iii) the relationship
between group cohomology and $\Ext$-groups which results by deriving the relationship
$\Hom_{G_2}(W, V) = (V \otimes W^* )^{G_2}$. 

Our statement \eqref{orig} is precisely analogous to the isomorphism of \eqref{orig.shapiro} with $\Gamma$ playing the role of $G_1$,
$G$ playing the role of $G_2$, and with topology inserted. It can be proven simply 
by writing down the explicit maps from far left to far right in \eqref{orig.shapiro}
and checking that they respect  topology and are inverse to one another.  There
is only one point that is not formal: to prove that
  there is a well-defined map from left to right, one needs to check
  that the extension of $G$-representations arising in (iii) by ``inflating'' a cocycle $j: \Gamma \rightarrow W^*_{-\infty}$ indeed has moderate growth. This requires
growth bounds on $j$, and these follow  simply
by writing out $j(\gamma)$, for arbitrary $\gamma \in \Gamma$,
in terms of the values of $j$ on a generating set using the cocycle relation.
We observe that some ``automatic continuity'' argument of this nature is needed, because,  in the statement
of \eqref{orig},  
  the topology of $W$ figures on the right hand side but not on the left. 
 
As the next step towards Theorem \ref{mainthm}, observe that there is a natural map 
   \begin{equation}  \label{extdef2} 
 \Ext_G^1(W_{\infty},\Abig) {\longrightarrow} \Ext^1_{\gK}(W, \Asmall),
\end{equation} 
where the right-hand side is taken in the category of $\gK$-modules.

This ``natural map'' associates to an extension $\Abig \rightarrow E \rightarrow W_\infty$
the underlying sequence of $K$-finite vectors in each of $\Abig, E, W_{\infty}$ which are  annihilated
by some power of $\Cas-\lambda$ (in the case of $W_\infty$, this space is exactly $W$, on which $C-\lambda$ acts trivially).
  That the resulting sequence remains exact follows from surjectivity of  $\Cas-\lambda$ in the form of 
Proposition \ref{casprop}. We explicate this: one must
verify that each element $w \in W$ 
has a preimage in $E_K$ killed by some power of $(\Cas-\lambda)$.
First, take an arbitrary preimage of $w$ in $E$ and average it over $K$ to produce a preimage $e \in E_K$. Then $(\Cas-\lambda)e$ belongs to the image of $\Abig_K$,
and can be written as $f_1+f_2$ with $f_1 \in \Cusp(\lambda) \subset \ker (\Cas-\lambda)$
and $f_2 \in \Cusp(\lambda)^{\perp}$. Choose, by Proposition \ref{casprop},
a class $e' \in \mA_K$ with $(\Cas-\lambda) e'= f_2$; then $e-e'$ still lifts $w$
and is now killed by $(\Cas-\lambda)$.

 We will show  in \S \ref{eval}  that the right hand side
of \eqref{extdef2} has dimension 
$$ m=  \mbox{ the multiplicity of $W^{cl}$ in $\Cusp(\lambda)$},$$
and in \S \ref{iso} that \eqref{extdef2} is actualy an isomorphism. 
This will conclude the proof,  remembering that 
  the left-hand side is identified, by means of \eqref{orig}, with $H^1(\Gamma, W^*_{-\infty})$.

\subsection{Evaluation of the $\gK$-ext}  \label{eval} 
 We compute the $\gK$-extension on the right-hand side of \eqref{eq cohoequalsgk}. 
 The space~$\Cusp(\lambda)$ decomposes as a finite 
  direct sum of irreducible~$\gK$-modules; this  follows from the similar~$L^2$ statement, see \cite[\S 8]{Borel97}.  
  Since each of these irreducible summands has infinitesimal character $\lambda$, their underlying $(\mathfrak{g},K)$-modules can belong to at most three isomorphism classes, as described in \S \ref{gKmodules};
among these is $W^{cl}$, the ``complementary $\gK$-module to $W$''
from Definition \ref{def complementary}.
Accordingly we decompose
\begin{equation}\label{eq Aut form decomposition} \Asmall = \Cusp(\lambda)^{\perp} \oplus (W^{cl})^{m} \oplus \bigoplus_{\stackrel{V \subset \Cusp(\lambda)}{V \not \simeq W^{cl}}} V, \end{equation} where 
 $\Cusp(\lambda)^{\perp}$
is the orthogonal complement of $\Cusp(\lambda)$ within $\Asmall$, and 
 $m$ is the multiplicity of $W^{cl}$ in $\Cusp(\lambda)$.

The splitting \eqref{eq Aut form decomposition} induces a similar direct sum splitting of $\Ext^1_{\gK}(W, \Asmall)$. 
But Proposition \ref{casprop} implies that
 $\Cas-\lambda$ defines a surjection from $\Cusp(\lambda)^{\perp}$ to itself,
 and so, applying Proposition \ref{comp}, 
$$ \Ext^1_{\gK}(W, \Cusp(\lambda)^{\perp}) = 0.$$
The  remaining two summands evaluate
via the second part of Proposition \ref{comp} to $\C^m$ and $0$ respectively. This yields 
$$\Ext^1_{\gK}(W, \Asmall) \simeq \C^m.$$
This concludes the proof that the right hand side of \eqref{extdef2} has dimension $m$. 

 \subsection{Comparison of topology and $\gK$ extensions} \label{iso} 
To conclude, we must show that the map of \eqref{extdef2} is in fact an {\em isomorphism}. 

Injectivity of the   resulting map on $\Ext$-groups
follows using the Casselman-Wallach theory of canonical globalization;
the result is formulated in exactly the form we need in \cite[Prop 11.2]{BK14}, namely, 
 a splitting at the level of $\gK$-modules automatically gives rise to a continuous splitting. \footnote{
 We sketch the idea of the argument to emphasize that what we use is relatively elementary: Given an abstract $\gK$-module
 splitting $\varphi: W \rightarrow \Abig$ we must show that it does not distort norms too far. 
Fixing a generating set $w_1, \dots, w_r$ for $W$, one shows using 
bounds similar to \eqref{mcbound} that
any vector $w \in W$ can be written as $\sum h_i \star w_i$ where $h_i$ are
bi-$K$-finite functions on $G$ and the norms of the $h_i$ are not too large in terms of the norms of $w$.
This permits one to bound the size of $\varphi(w) = \sum h_i \star \varphi(w_i)$. 
 }
 
 For surjectivity, one cannot directly apply the Casselman--Wallach theory because  $\mA$ is ``too big.''
 However, we saw in \S \ref{eval} that  the right-hand side of \eqref{extdef2} 
actually is generated by the image of $\Ext^1_{\gK}(W, \Cusp(\lambda))$.
The space $\Cusp(\lambda)$ has finite length, and 
then the results of \cite{Ca89} (in the form of the equivalence of categories, see \cite[Corollary, \S 11.6.8]{Wa92}) implies 
that each such extension of $\gK$-modules arises from an extension of smooth globalizations, which readily implies the desired surjectivity.

\begin{remark} \label{higherFrobenius}
 Together, the isomorphisms  \eqref{orig} and \eqref{extdef2} give an isomorphism \begin{equation} \label{eq cohoequalsgk} H^1(\Gamma, W^*_{-\infty}) \simeq 
 \Ext^1_{\gK}(W, \Asmall).\end{equation}
  The analogous statement  in {\em all} cohomological degrees
 has been proved for cocompact $\Gamma$ by Bunke-Olbrich \cite[Theorem 1.4]{BO97}.  
    However,  our argument does not generalize to this case, at least in any routine way:
    it is not immediately clear to us how to generalize the cocycle growth
  argument to $H^i$ for $i > 1$. 
  \end{remark}
  
  \begin{remark}\label{FrobRepProof}
 For completeness, because  we made use of it earlier,  we outline the argument 
 for the much easier degree $0$ version of \eqref{eq cohoequalsgk}, i.e. 
 ``Frobenius reciprocity'':
\begin{equation} \label{FR0} H^0(\Gamma, W^*_{-\infty}) \simeq \Hom_{\gK}(W, \Asmall) ,\end{equation} 
 where we now allow $W$ to be any finite length $\gK$-module. 
 
 The standard construction of Frobenius reciprocity
identifies $H^0(\Gamma, W^*_{-\infty})$ with continuous $G$-homomorphisms
from $W_\infty$ to $\Abig$; then, restriction to $K$-finite vectors defines
a class in $\Hom_{\gK}(W_K, \Abig) \simeq \Hom_{\gK}(W_K, \Asmall)$.
This restriction
map is an isomorphism by the Casselman--Wallach theory \cite[Theorem, \S 11.6.7]{Wa92},
taking the target space to be the subspace
of $\Abig$ comprising functions which are   (i)
by  killed by a fixed large power of $(\Cas-\lambda)$ and (ii)
have finite norm \eqref{uniform moderate growth} 
for all $X$ and for some fixed large $R$.  This proves \eqref{FR0}. 

Now \eqref{FR0} implies Lemma \ref{FrobRep}:
$W$ is  is annihilated by an ideal of finite codimension in $Z(\mathfrak g)$; as such, the image of any 
$\gK$-homomorphism from $W$ to moderate growth functions automatically has image inside functions of uniform moderate growth \cite[5.6]{Borel97},
and therefore has image in $\Asmall$.

\end{remark}

\newcommand{\Discrete}{\mathrm{Discrete}}
\section{Surjectivity of Casimir on the space of automorphic forms.} \label{ourproof}
The primary analytic ingredient
in both proofs is the surjectivity of a Laplacian-type operator;
in the first proof this is used on spaces of functions both on $G$ and on $\Gamma \backslash G$,
and in the second proof it is used only on $\Gamma \backslash G$. 
We will now give a self-contained proof of the second version, Proposition \ref{casprop}. 
As noted after that proposition, this statement is essentially in the literature, but given
its importance it seemed appropriate to give a self-contained proof.

We follow here the notation of \S \ref{section our proof}; in particular, $\mA$
is defined using the notion of {\em uniform} moderate growth.  
It is enough to show that every
function orthogonal to $\Cusp(\lambda)$ occurs in the image of~$\Cas-\lambda: \mA_K \rightarrow \mA_K$. 
The basic
strategy is as follows: 
\begin{itemize}
	\item[(i)] In \S \ref{ss decomposition}, we decompose elements of~$\mA_K$ into functions ``near the cusp'' and functions of rapid decay,  and 
	\item[(ii)] in \S \ref{ss surjectivity},  we construct preimages under $\Cas-\lambda$ for functions in each subspace.  
	Doing this ``near the cusp''  	amounts to solving an ODE; 
	 the construction of preimages for functions of rapid decay is  carried out 
	 via $L^2$-spectral theory. 
\end{itemize}

 Since~$\Cas-\lambda$ commutes with~$K$, it suffices to prove Proposition \ref{casprop} with $\mA_K$ replaced by its subspace
$\mA_{\zeta}$ with ~$K$-type $\zeta$.  In what follows, we will regard $\zeta$ as fixed.

\subsection{Cusps} \label{ss cusps}
 It is convenient to fix once and for all a fundamental domain for $\Gamma \backslash G$: we take 
\begin{equation} \label{Frakfdef} \mathfrak{F} = \{z \in \mathbb{H}: d(z,i) \leq d(\gamma z, i) \mbox{ for all } \gamma \in \Gamma -\{e\}\}, \end{equation} 
 which  
 describes a convex hyperbolic polygon which is (up to boundary) a fundamental domain for $\Gamma$
 acting on $\mathbb{H}$; its pullback to $G$ via $g \mapsto g \cdot i$ is a fundamental domain for
 $\Gamma \backslash G$, which will often be denoted by the same letter. 
 In particular, $\mathfrak{F}$ can be decomposed in the following way, where the sets intersect only along their boundary: 
\begin{equation} \label{Fdecomp} \mathfrak{F} = \mathfrak{F}_0 \cup \mathcal{C}_1 \cup \mathcal{C}_2 \cup \dots \cup \mathcal{C}_h\end{equation}
  with $\mathfrak{F}_0$ compact  and each $\mathcal{C}_i$ a {\em cusp}, that is to say,
 a $G$-translate of a region of the form $\{x+iy: a \leq x \leq b, y \geq Y_0\}.$ 
 In the Iwasawa cordinates $G= NAK$ of \eqref{eq Iwasawa}, the pullback of~$\mathcal{C}_i$ to $G$ therefore has the form 
 \begin{equation} \label{Siegel-domain}
\widetilde{\mathcal{C}_i} =  g_i \cdot \{n_x a_y k: a \leq x \leq b, y \geq Y_0, k \in K\}.\end{equation}
The map $\widetilde{\mathcal{C}_i} \rightarrow \Gamma \backslash G$ is injective
on the interior of $\widetilde{\mathcal{C}_i}$. We will often
identify $\widetilde{\mathcal{C}_i}$ with its image in $\Gamma \backslash G$.

\subsection{The constant term and moderate growth functions in the cusp} \label{ss constant term}

 Let $f \in \mA_\zeta$. Fix a cusp $i$; we write $\Gamma^N_i$ for $\Gamma \cap g_i N g_i^{-1}$. 
 The constant term 
 $f^N_i:   g_i N g_i^{-1} \backslash G \longrightarrow \C$ is defined by the rule
 \begin{equation}\label{eq constant term def}
 	f^{N}_i :  x \mapsto  \mbox{average value of $f(g_i n_t g_i^{-1} x)$ for $t \in \mathbf{R}$.} \end{equation}
 The function $f(g_i n_t g_i^{-1} x)$ is periodic in $t$
 and therefore the notion of its average value makes sense. Moreover,
 the above map is right $G$-equivariant. 
 A basic (and elementary) fact is that $f^N_i$ is asymptotic to $f$ inside $\widetilde{\mathcal{C}}_i$; indeed
the function $f-f^N_i$ has rapid decay in  $\widetilde{\mathcal{C}_i}$,
  as proved in  \cite[7.5]{Borel97}. 
 Here, 
we say that a function $J: \widetilde{\mathcal{C}_i} \rightarrow \C$
 has {\em rapid decay} if, for any $X_1, ...,  X_r \in \mathfrak{g}$
 and any positive integer $N$ we have
 \begin{equation} \label{eq rapid decay} \sup_{\widetilde{\mathcal{C}_i}} \|g\|^{N} \left|X_1 \dots X_r J(g)\right| < \infty. \end{equation}

Let us consider more generally functions $f$ on $G$ that are left $N$-invariant and have fixed right $K$-type $\zeta$.  
Such a function may be identified, by means of pullback by $y \mapsto a_y$, with a function $f$ on~$\R_+$. 
 The condition of  the original $N$-invariant function on on $G$ having finite norm under~$\|\cdot\|_{X,R}$ for all~$X$,
 with notation as in \eqref{uniform moderate growth}, 
is equivalent to asking that 
\begin{equation}
 \label{noddy} \left| (y \frac{d}{dy})^j f \right| <    C_j \cdot (|y|^{-1}+|y|)^R\end{equation} for all $j$. 
 That this condition is necessary is seen by applying \eqref{uniform moderate growth} 
 to $X$ a product of elements in $\mathrm{Lie}(A)$. 
 To see that it is sufficient, we 
fix $\mathfrak{U}$ belonging to the universal enveloping algebra of $\mathfrak{g}$;
 now, for any $k \in K$, we may write $\mathfrak{U} $ as a sum of terms $\sum c_i(k) (\Ad(k^{-1}) \mathfrak{U}_{N,i}) ( \Ad(k^{-1}) \mathfrak{U}_{A,i}) \mathfrak{U}_{K,i}$
 where  the terms belong to fixed bases for the universal enveloping algebra of $N,    A$ and $K$ respectively, and the coefficients $c_i(k)$ are bounded independently of $k$.
 This permits us to bound $\mathfrak{U}f(nak)$ and we see that the bound \eqref{noddy} suffices. 
  
This motivates the following definition:
Fix $Y_0 > 0$ and denote by $  \mathcal{P}_{\geq Y_0}$
the space of smooth functions on $\R$ supported in $y > Y_0$ 
satisfying \eqref{noddy} for some $R$. Because of the restriction that $y>Y_0$,
this  is equivalent to ask that 
all derivatives are   ``uniformly'' polynomially bounded, i.e. there is $R$ such that, for all $j$,  %
there exists a constant $C_j$ with
\begin{equation} \label{growthbound}  \left| \frac{d^j f}{dy^j}  \right| <  C_j (2+|y|)^{R-j}.
\end{equation}

\subsection{The subspace $\Eis_{\lambda}$ of Eisenstein series
with eigenvalue $\lambda$} \label{Eis sec}
To each cusp $\mathcal{C}_j$, we attach an Eisenstein series $E^j(s)$, 
which is an $\mA_{\zeta}$-valued meromorphic function of the complex variable $s$, 
characterized by the fact that for $\mathrm{Re}(s) \gg 1$ it equals: 
\[E^j(s, g) = \sum_{\gamma \in \Gamma_N^i \backslash \Gamma} H(g_i^{-1} \gamma  g)^s,\]
where $H$ is  
the  
 unique  function on $G$ with right $K$-type $\zeta$,
  invariant on the left by $N$, and on $A$ given by $a_y \mapsto y$. 
  
The resulting vector-valued function is holomorphic when $\mathrm{Re}(s) = 1/2$
and we denote its value at $s=1/2+it$ by $E_t^j$. In words, $E_t^j$ is 
 the unitary Eisenstein series of $K$-type $\zeta$ with parameter $t \in \R$
attached to the $j$th cusp of $\Gamma \backslash G$. 
 Finally, denoting by $\lambda_t$ the eigenvalue of $\Cas$ on $E^j_t$, let
\[ \Eis(\lambda) := \bigoplus _j \left\{\mbox{ span of all Eisenstein series~$E^j_t$, with $t \in \R$, such that $\lambda_t={\lambda}$}\right\},\] so that~$\Eis(\lambda)$  is a finite-dimensional subspace of~$\mA_{\zeta}$ annihilated by $\Cas-{\lambda}$. 
However, if the quadratic function $t \mapsto \lambda_t-{\lambda}$ happens to have a double zero,
we include in the above space the derivative $\frac{d}{dt} E^j_t$, for this
is also annihilated by $\Cas-{\lambda}$. 
 The Casimir eigenvalue of~$E^j(s,g)$ is quadratic in~$s$ and therefore the dimension of $\Eis(\lambda)$ is at most
 twice the number of cusps.

\subsection{Decomposition of $\mA_\zeta$} \label{ss decomposition} 

Consider the subspace of $\mA_{\zeta}$ consisting of  $L^2$-eigenfunctions of the Casimir with eigenvalue $\lambda$; call 
 this $\mathrm{Discrete}(\lambda)$.   
 
 \begin{lemma} \label{cutoffLemma}
	Let $\widetilde{\mathcal C}_i$ be the cusps for a fundamental domain for the action of $\Gamma$ on $G$ as in \eqref{Fdecomp}.  Then every $f \in \mA_{\zeta}$, perpendicular to $\Cusp(\lambda)$, can be written as the sum \begin{equation} \label{eq decomposition}f = f_s + \sum_{i} f_{c_i}\end{equation}  where: \begin{itemize}
		\item[(i)] The function $f_s$ is smooth, has rapid decay at all the cusps, and is perpendicular to $\Eis(\lambda) \oplus \mathrm{Discrete}(\lambda)$.  
		\item[(ii)] Each $f_{c_i}$ is supported in the cusp $\widetilde {\mathcal{C}}_i$ and,
		with reference to the identification~\eqref{Siegel-domain}: \[ \widetilde{ \mathcal{C}_i }=  g_i \cdot \{n_x a_y k: a \leq x \leq b, y \geq Y_0, k \in K\}. \] 
		has the form \begin{equation} \label{nak function} n_xa_yk \mapsto P(y) \zeta(k), \end{equation} where  $P$  
		belongs to the space $\mathcal{P}_{\geq Y_0}$ described   after  
		 \eqref{growthbound}.
	\end{itemize} 
\end{lemma}

  Observe that, although $f$ is only assumed orthogonal to cusp forms, we arrange that $f_s$ is orthogonal also to $\Eis(\lambda)$ and all of $\Discrete(\lambda)$.
  This is possible because there is a lot of freedom in the decomposition \eqref{eq decomposition}.  It will be very convenient later.
  
 \proof 
This is a straightforward cut-off process; the only delicacy is to ensure that~$f_s$ is in fact perpendicular to $\Eis(\lambda)$ and $\Discrete(\lambda)$. 
We start from $f^{N}_i$, the constant term along  the $i$th cusp as defined in \eqref{eq constant term def}. 
Take $\varphi_i, \psi_i$ smooth functions on $\mathbf{R}_+$  where: 
\begin{itemize}
	\item $\varphi_i=0$  for $y<Y_0$ and $\varphi_i = 1$ for $y > 2Y_0$. 
	\item $\psi_i$ is supported in $(Y_0, 2Y_0)$. 
\end{itemize}
We consider both $\varphi_i$ and $\psi_i$ as functions on $\widetilde{\mathcal{C}}_i$ 
described by the rules
 $g_i n_x a_y k \mapsto \varphi_i(y) $ and  $g_i n_x a_y k \mapsto \psi_i(y) \zeta(k)$ respectively. 
 Now put $f_s = f - \sum_{i} ( \varphi_i f^N_i+\psi_i)$  so that 
\begin{equation} \label{fg} f =  f_s + \sum \underbrace{  \left( \varphi_i f^N_i + \psi_i \right)}_{f_{c_i}}.  \end{equation}
We will show that, for suitable choice of $\psi_i$, 
\eqref{fg} is the desired splitting of $f$. 
All the properties except perpendicularity to $\Discrete(\lambda) \oplus \Eis(\lambda)$ follow from general properties of the constant term discussed in \S \ref{ss constant term}.   In particular, the uniform bound on 
the functions $P$ associated -- as in \eqref{nak function} -- to the various $f_{c_i}$ follow from the condition that $f$ has uniform moderate growth.

Observe that $\varphi_i f^N_i$  and $\psi_i$ are both perpendicular to
all cuspidal functions and in particular to $\Cusp(\lambda)$,
because they both arise from functions on $g_i N g_i^{-1} \cap \Gamma \backslash G$
which are  left invariant by $g_i N g_i^{-1}$. 
Therefore
$f_s$ is also perpendicular to $\Cusp(\lambda)$. 

It remains to  choose  $\psi_i$ in such a way  that $f_s$ is indeed perpendicular to 
the orthogonal complement of $\Cusp(\lambda)$ inside $\Discrete(\lambda) \oplus \Eis(\lambda)$;
call this space $\widetilde{\Eis}(\lambda)$, as it is (potentially) a finite-dimensional enlargement of $\Eis(\lambda)$. 
 To do this, for each $\mathcal{E} \in \widetilde{\Eis}(\lambda)$ we  should have
\[  \left\langle \sum_{i} f - \varphi_i f^N_i, \mathcal{E} \right \rangle =  
\sum_{i} \left\langle \psi_i, \mathcal{E}^N_i \right \rangle_{\widetilde{\mathcal{C}_i}}.\]
The right-hand side can be considered as a linear mapping from the vector space
of possible $\psi_i$ to the finite-dimensional dual $\widetilde{\Eis}(\lambda)^*$ of the vector space $\widetilde{\Eis}(\lambda)$. 
It is enough to show this mapping is surjective, and for this it is enough to show
that its dual is injective. But the dual map is identified
with the constant term: 
\[\widetilde{\Eis}(\lambda) \rightarrow \bigoplus_{i} C^{\infty}(T_i, 2T_i), \quad  \mathcal{E} \mapsto (\mathcal{E}^{N}_i)(g_i a_y)\]
and this is injective: if $\mathcal{E}^{N_i}$
vanished in $(T_i, 2T_i)$ then it  -- being real-analytic -- vanishes identically;
  if this is so for all $i$, then $\mathcal E$ would be a cusp form, contradicting the definition of $\widetilde{\Eis}(\lambda)$. 
 \qed

\subsection{Surjectivity of $\Cas - \lambda$} \label{ss surjectivity}

We now show surjectivity of $\Cas-\lambda$ on each of the two pieces of $\mA_\zeta$
corresponding to the decomposition of Lemma   \ref{cutoffLemma}. 

\subsubsection{Surjectivity on the cusp}

\begin{lemma} \label{variation} 
The operator~$\Cas-\lambda$ is surjective on the space of functions on $G$ 
	which:
	\begin{itemize} \item are left $N$-invariant and have fixed right $K$-type $\zeta$, and
	\item lie in the space  $\mathcal{P}_{\geq Y_0}$  described before \eqref{growthbound} when pulled back to $\mathbf{R}_+$ by means of $y \mapsto a_y$.  
	\end{itemize}
 
	\end{lemma}
\begin{proof}
Let $f: \mathbf{R}_+ \rightarrow \C$ 
 be extended to  a function $F: G \rightarrow \C$ by left $N$-invariance and with 
fixed $\kappa$-weight equal to $\zeta$, so that $F$ has the form:  \[F(na_y \exp(\theta k)) = f(y) e^{i \zeta \theta}.\]
 Observe 
that for arbitrary $X_1 \in \mathfrak{n} = \mathrm{Lie}(N)$ and $X_2, \dots, X_k \in \mathfrak{g}$ we have
\[\left( X_1 \dots X_k F  \right) \mbox{ is identically zero on $NA$.}\]
Indeed, the left-hand side is the partial derivative $\partial_{t_1} \dots \partial_{t_k}$ of $F(na e^{t_1 X_1} \dots e^{t_k X_k})$
evaluated at $t_i=0$, which vanishes since $F$ is independent of $t_1$.
From this observatoin, it follows that the action of the operator 
  $\Cas=   \frac{H^2}{2} - H + 2 XY$ on $f$ agrees with the action of 
 $ H^2/2 - H$ on $f(y)$. Since
$H$ acts on $f$ via via $2 y \frac{d}{dy}$, we get that $\Cas -\lambda$ acts as the differential operator: 
\[  2 y^2 \frac{d^2}{dy^2} - \lambda.\]
 
 We show that $\Cas-\lambda$ is surjective on
 $\mathcal{P}_{\geq Y_0}$ explicitly:
we construct $g$ with $(\Cas-\lambda) g=f$
via the method of variation of parameters. 
   
 The homogeneous solutions 
 to the equation $\left(2 y^2 \frac{d^2}{dy^2}-\lambda \right) g = 0$
 are given by~$y^{p_1}$,~$y^{p_2}$,
 where the $p_i$ are roots of $2 p(1-p) + \lambda=0$. We assume that $p_1 \neq p_2$, the $p_1 = p_2$ case is similar. 
  A solution to $(\Cas-\lambda) g=f$ can then be found by taking
 \[g= b_1(y) y^{p_1} + b_2(y) y^{p_2},\]
  where the $b_i$ satisfy 
 \[ \frac{db_i}{dy} = (-1)^i \frac{1/2}{p_1-p_2} f(y) y^{-p_i-1} .\] 
 Taking $f$ as in \eqref{growthbound}, we take
 $b_i = \pm \frac{p_i-p_2}{2} \int_{Y_0}^y f(y) y^{-p_i-1}$
 for $y > Y_0$ and $b_i(y) = 0$ for $y \leq Y_0$. By construction, if $f$ belongs to $\mathcal{P}_{\geq Y_0}$
 then so does $b_i$ and so also $g$.     \end{proof}

\subsubsection{Surjectivity on functions of rapid decay.}

\begin{prop} \label{spectral}
The image of the map~$\Cas-\lambda: \mA_\zeta \to  \mA_\zeta$ contains all functions of rapid decay  	that are orthogonal to $\Eis(\lambda)$ and $\Discrete(\lambda)$. 
\end{prop}

\begin{proof}

Let $f$ be such a function.   We
fix an orthonormal basis $\{\varphi_i\}$ for 
the discrete spectrum of $\Cas-\lambda$ on $L^2(\Gamma \backslash G)_{\zeta}$,
where the subscript means that we restrict to $K$-type~$\zeta$.
For constants $\mu_j$ depending only on the width of the various cusps, we have, following e.g. \cite[\S 13]{Borel97}, 
\begin{equation} \label{fspectral} f = \sum_{i} \langle f, \varphi_i  \rangle \varphi_i +  \mu_j \sum_j \int_{t \geq 0} \langle f, E_t^j \rangle E_t^j  dt.\end{equation}
 {\em A priori} this is an equality inside $L^2$.
 Let $\lambda_i$ and $\lambda_t$ be, respectively, the eigenvalues of $\Cas-\lambda$ on $\varphi_i$ and $E_t$; 
 by the assumption on $f$, these are nonvanishing except when $\langle f, \varphi_i \rangle= 0$ or when $\langle f, E_t \rangle=0$.

Define $\bar{f} \in L^2$ by the rule  \begin{equation} \label{barf-def} \bar{f} = \sum_{\lambda_i \neq 0} \frac{ \langle  f, \varphi_i \rangle}{\lambda_i} \varphi_i +  \sum_{j} \int_{t \in \R} \frac{ \langle f, E_t^j \rangle}{\lambda_t} E_t^j dt.\end{equation} 
 
It is not hard to see that the right-hand side defines an $L^2$-function: 
 The function $\langle f, E_t^j \rangle$ is holomorphic in a neighbourhood of $t \in i\R$, as follows
 from holomorphicity of $t \mapsto E_t^j$ and absolute convergence of the integral defining $\langle f, E_t^j \rangle$. 
 Moreover, by assumption, 
this holomorphic function vanishes when $\lambda_t=0$.
 In particular the function $\langle f, E_t^j \rangle/\lambda_t$ is holomorphic, too;
 this follows from what we just said if the quadratic function $t \mapsto \lambda_t$
 has distinct zeroes, and in the case when it has a double zero $t_0$ we recall that the derivatives
 $\frac{dE_t^j}{dt}|_{t=t_0}$ also belong to $\Eis(\lambda)$. 
Therefore, the integrand in  \eqref{barf-def} is  locally integrable in $t$, and 
then its global integrability follows from \eqref{fspectral}.

We claim that  $\bar{f}$ has uniform moderate growth and $$(\Cas-\lambda) \bar{f} =f$$
as desired.

 In fact,   the summation and integrals in both \eqref{fspectral} and \eqref{barf-def} are
 absolutely convergent, uniformly on compact sets, and they define
functions of uniform moderate growth; moreover,  any derivative $X \bar{f}$
coincides with the corresponding summation inserting $X$ inside the sums and integrals.
 The proof of these claims follow from  nontrivial, but relatively  standard, estimates. We summarize these estimates, with references.
    A convenient general reference for all the analysis required is Iwaniec \cite{Iwaniec};
 he works only with the trivial $K$-type, but analytical issues
  are exactly the same if we work with a general $K$-type. 
    
    We examine the first summand of \eqref{barf-def} first. 
 Let $\lambda_i$ be the  $(\Cas-\lambda)$-eigenvalue of $\varphi_i$. 
 Then the easy upper bound in Weyl's law (cf. \cite[(7.11), Corollary 11.2]{Iwaniec}
 for the sharp Weyl law in the spherical case; the same proof applies with $K$-type) gives:  \begin{equation} \label{decay0} \# \{i : |\lambda_i| \leq T \} \leq \mathrm{const} \cdot T^2.\end{equation}
 For any $r \geq 0$ we have an estimate
\begin{equation} \label{decay1}  |\langle f, \varphi_i \rangle| \leq c_r (1+|\lambda_i|)^{-r}
\end{equation} 
 arising from 
  integration by parts and Cauchy-Schwarz (using $\|\varphi_i\|_{L^2}=1$). 
Finally, there is a constant $N$ with the following property: for any invariant differential operator $X \in \mathfrak{U}$ 
 of degree $d$, we have a bound
\begin{equation} \label{decay2} | X \varphi_i (g)|  \leq  (1+|\lambda_i|)^{d+N} \|g\|^N. \end{equation} 
This can be derived from a Sobolev estimate, again using the normalization $\|\varphi_i\|_{L^2}=1$; see
e.g \cite[(3.7)]{BRSobolev}.    These estimates suffice
 to treat the cuspidal summand of \eqref{barf-def}.

  Now we discuss the integral summand of \eqref{barf-def}.     
 To examine absolute convergence of the integral, one  reasons exactly as for cusp forms, but 
rather than pointwise estimates in $t$ one only looks at averages over $T \leq t \leq T+1$. 
In place 
of the $L^2$-normalization of $\varphi_i$ we have the following estimate
$$  \int_{T}^{T+1}  \int_{\mathrm{ht} \leq Y} |E^j_t(g)|^2  \ll T^2 + \log(Y)$$
where $\mathrm{ht} \leq Y$ means that we integrate over the complement
of the set $y \geq Y$ in each cusp. 
This bound is derived from the Maass--Selberg relations  (cf. \cite[Proposition 6.8 and (6.35) and (10.9)]{Iwaniec})
and average bounds on the scattering matrix (equation (10.13), {\em op. cit.}). 
From this, one obtains in the same way as the cuspidal case bounds on 
    $\int_{T}^{T+1}  | \langle f, E^j_t \rangle |^2$
     and $\int_{T}^{T+1} |X E^j_t|^2$
     that are of the same quality as    \eqref{decay1} and \eqref{decay2}
     and the same analysis as for the cuspidal spectrum goes through. 

 \end{proof}

 \subsection{Proof of the proposition}
 
We now prove Proposition \ref{casprop}, that is to say, that the image of $\Cas-\lambda$ is the orthogonal complement of cusp forms.   Take $f \in \mcA_{\zeta}$ and write $f = f_s + \sum f_{c_i}$ as in Lemma \ref{cutoffLemma}. By Lemma \ref{variation}
and Proposition \ref{spectral} there are functions
$g_i, g \in \mA_{\zeta}$ with
$$  (\Cas-\lambda) g_i = f_{c_i}, \quad (\Cas-\lambda) g=f_s,$$  
where, in the case of $g_i$, we use  Lemma \ref{variation} to produce a function on $\widetilde{\mathcal{C}_i}$, and then extend it by zero to get 
 an element of $\mA_{\zeta}$.  Then $g+\sum_i g_i$ is the desired preimage of~$f$ under $\Cas-\lambda$. 
\qed

\section{Interpolation and cohomology.} 

\label{InterpCoh}

 We will recall background on the Segal-Shale-Weil representation
 (see \cite{LV13} for details) necessary to explain 
  why the foregoing results imply the interpolation formula of Radchenko and Viazovska. 
We have already outlined the argument in \S \ref{outline}
and what remains is to explain in detail where the actual numbers   in \eqref{win} come from. 

\subsection{The Weil representation} \label{Weil}

 Let $L^2(\R)_+$ be the Hilbert space of even square integrable functions on $\R$, and
let $\mS$ be the subspace of even  Schwartz functions, i.e. even 
smooth functions $f$ such that 
\begin{equation} \label{Sdef} \sup_{x \in \BR}\left | x^n \frac{d^m}{d x^m} f(x) \right | <\infty  \end{equation}  for any pair $(m,n)$ of non-negative integers. 
Let $G$ be the degree $2$ cover of $SL_2(\BR)$. There is a unique unitary representation of $G$ on $L^2(\R)_+$, the Weil (or oscillator) representation,
for which $\mS$ is precisely the subspace of smooth vectors
and such that the action of
 $\fg$ on $\mS$ is given by: \[ X \cdot \phi(x)= -i\pi x^2\phi(x), \qquad Y \cdot \phi(x) = \frac{-i}{4 \pi}\frac{\del^2}{\del x^2}\phi(x), \qquad H \cdot \phi(x) =  \left( x\frac{d}{dx} + \frac{1}{2}\right)\phi(x). \]  
   It then follows that $\kappa = i(X-Y)$ acts by \[ \kappa \cdot \phi(x) = \left(\pi x^2 - \frac{1}{4\pi}\frac{\del^2 }{\del x^2}\right) \phi(x) \]

The normalization ensures that the action of $G$ is unitary and that the relation~$\sigma X \sigma^{-1} = Y$ is preserved, where~$\sigma: \mS \to \mS$ is the Fourier transform: \[ \sigma(\phi)(\xi) = \hat{\phi}(\xi) := \int_\BR \phi(x)e^{-2\pi i x \xi} dx. \]
Moreover,  with respect to the seminorms of   \eqref{Sdef}, 
the topological vector space $\mathcal{S}$ has the structure of a moderate growth Fr\'echet representation of $G$. 

The vector $v_{1/2} := e^{-\pi x^2}$ has $\kappa$-weight $1/2$
and Casimir eigenvalue $-3/8$.
The other $K$-finite vectors in $\mathcal{S}$
are spanned by its Lie algebra translates; they have the form~$q(x) e^{-\pi x^2}$ for $q$ an even polynomial, and have $\kappa$-weights
$\frac{1}{2}, \frac{5}{2}, \frac{9}{2}, \cdots$.

\subsection{The lattice $\Gamma$}

 If $X \in \fg$ is nilpotent, the projection map identifies~$\exp(\BR X) \subset \meta$
with the corresponding $1$-parameter subgroup of $SL_2(\BR)$. In particular,
the map~$G \rightarrow SL_2(\BR)$ splits over any one-parameter unipotent subgroup;
thus  the groups of upper and lower-triangular matrices have distinguished lifts in $\meta$.

In particular,  the elements $e = \left(\begin{smallmatrix}
	1&2\\ 0&1
\end{smallmatrix}\right)$ and $f = \left(\begin{smallmatrix}
1&0\\ 2&1
\end{smallmatrix}\right)$ defined in \eqref{efdef} have distinguished lifts $\tilde{e}$, $\tilde{f}$ to $G$. They act in the Weil representation by:
\begin{equation} \label{eq unipotent action} \tilde{e} \cdot \phi (x) = e^{-2 \pi i x^2} \phi(x), \qquad \tilde{f} \cdot \phi(x) = \sigma \tilde{e} \sigma^{-1} \phi(x). \end{equation}

Let $\Gamma \in SL_2(\BZ)$ be the subgroup freely generated by $e$ and $f$. 
 It is the subgroup of $\Gamma(2)$ whose diagonal entries are congruent to $1$ mod $4$, and is conjugate to $\Gamma_1(4)$. 

\begin{lemma}
	There is a splitting~$\Gamma \to \meta$ which extends the splitting over the two subgroups~$\langle e \rangle$ and~$\langle f \rangle$. The image of~$\Gamma$ in this splitting are precisely the elements of its preimage leaving fixed the 
	distribution~$\mathcal{Q} := \sum_{n \in \BZ} \delta_{n^2}$ -- see \ref{outline} for the definition of $\delta_n$. 
\end{lemma}
\begin{proof}
	The lift~$\tilde{e}$ of $e$ to $G$ fixes~$\mathcal{Q}$. By Poisson summation, so does the lift~$\tilde{f}$ of~$f$. The group $\tilde{\Gamma}$ generated by $\tilde{e}$ and $\tilde{f}$ surjects onto $\Gamma$ with kernel of size at most two. But $\tilde{\Gamma}$ fixes $\mathcal{Q}$, and the two lifts of any $g \in SL_2(\BR)$ to $\meta$ act on $\mS$ by different signs, so the map $\tilde{\Gamma} \to \Gamma$ is injective. 
\end{proof}

 \subsection{Conclusion of the proof} \label{ss End of the proof}

We now fill in the deduction, already sketched in the introduction, of the Interpolation Theorem \ref{interpolation}
from Theorem \ref{mainthm}. 

We first handle a detail of topology from the discussion of \S \ref{outline}, namely, the equivalence
between the interpolation statement and its ``dual'' form. 
For a Fr\'echet space $F$ we denote its continuous dual  by $F^*$;  we regard it as an abstract vector space without topology.  
Then, for $\eta: E \to F$ a continuous map of Fr\'echet spaces,
\begin{equation} \label{Frechetdual} \mbox{if $\eta^*: F^* \to E^*$ is bijective, then $\eta$ is 
 	a homeomorphism.}
	\end{equation} 
Indeed, following \cite[Theorem 37.2]{Treves}, 
a continuous homomorphism~$\eta: E \rightarrow F$
 of Fr\'echet spaces is surjective if
 $\eta^*$ is injective and its image is weakly closed. 
 Applying this in the situation of \eqref{Frechetdual}, 
 we see at least that $\eta$ is surjective. It is injective 
  because
 the image of~$\eta^*$ is orthogonal to the kernel of~$\eta$, and then we apply the open mapping theorem to see that it is a homeomorphism. 

To verify the equivalence, asserted in \S \ref{outline}, 
between Theorem \ref{interpolation} and its dual version, 
 we apply  
\eqref{Frechetdual} to the map
 $\Psi$ of Theorem \ref{interpolation}, with codomain
 the closed subspace of $\mathfrak{s} \oplus \mathfrak{s}$ 
  defined by $\sum_{n \in \Z} \phi(n) = \sum_{n \in \Z} \widehat{\phi}(n)$. 

The other point that was not proved in \S \ref{outline} was \eqref{win},  the actual evaluation
of $H^0$ and $H^1$ for the dual of the oscillator representation, namely
\begin{equation} \label{win2} \dim H^0(\Gamma, \mS^*)  =1, \quad \dim H^1(\Gamma, \mS^*) = 0.\end{equation} 
Now, $\mS^*$ is precisely the distribution globalization of the dual of $\mS_K$, i.e.
it is the~$W_{-\infty}^*$ of the statement of Theorem~\ref{mainthm} if we take
$W$ to be $\mS_K$.  
 Therefore Theorem~\ref{mainthm} reduces us to 
 showing that the multiplicity of $\mS_K$ (respectively $\mS_K^{cl}$) in the space of automorphic forms (respectively cusp forms) for $\Gamma$
equals $1$ (respectively $0$).

From \ref{Weil}, the $K$-finite vectors $\mS_K$
are a realization of the $\gK$-module of lowest weight $1/2$, whose complementary representation
$(\mS_K)^{cl}$ is the~$\gK$-module of highest weight~$-3/2$. 
In general, a homomorphism from a lowest weight~$\gK$-module to any~$\gK$-module~$W$
is uniquely specified by the image of the lowest weight vector, which can be an arbitrary element of $W$
killed by $m$; and the dual statement about highest weight modules is also valid. 

It follows that $\gK$-homomorphisms
from~$\mS_K$ (respectively~$\mS_K^{cl}$)
to the space~$\Abig$ of automorphic forms correspond exactly to holomorphic forms of weight $1/2$ (respectively, antiholomorphic forms of weight $-3/2$);
the conditions of being killed by $m$ or $p$ precisely translate to being holomorphic or antiholomorphic.  The desired
conclusion \eqref{win2} now follows from: 

\begin{lemma} \label{lemma spaces of mf}
\begin{itemize}
\item[(a)]
The space of holomorphic forms for $\Gamma$ of weight $1/2$ is one-dimensional,
and the space of cuspidal holomorphic forms of this weight is trivial.
\item[(b)]
The space of {\em cuspidal}  holomorphic forms for $\Gamma$ of weight $3/2$  is trivial;
therefore, the space of cuspidal antiholomorphic forms for $\Gamma$ of weight $-3/2$ is also trivial. 
\end{itemize}

\end{lemma}
\proof
For (a), the group $\Gamma$ is conjugate to $\Gamma_1(4)$, for which the space of modular forms of weight $1/2$ is spanned by the theta series  $\theta_{1/2}(z) = \sum_{n \in \BZ}e^{2\pi i z  n^2}$ \cite{SS77}. 
 
For (b), we use the fact that multiplication by $\theta$ injects the space of weight $3/2$ forms into the space of weight $2$ forms.  The space of weight $2$ cusp forms for $\Gamma_1(4)$ is, however, trivial;
indeed, the compactified modular curve $X_1(4)$ has genus zero. The final assertion follows by complex conjugation. 
\qed

\subsection{Variants: odd Schwartz functions, higher dimensions, Heisenberg uniqueness.}\label{triangle groups} 
We now show how the same ideas give several other interpolation theorems
 without changing the group $\Gamma = \langle e,f\rangle$;
it may also be of interest to consider $(\infty, p, q)$-triangle groups.

\subsubsection{Odd Schwartz functions} The discussion of Section \ref{Weil} on the even Weil representation $\mS$ carries verbatim to its odd counterpart $\mT$, whose $\gK$-module of $K$-finite vectors is spanned by the translates of the lowest weight vector $v_{3/2} = xe^{-\pi x^2}$.
As above, we compute  using Theorem \ref{mainthm}, to get \[ H^0(\Gamma, \mT^*) = \BC, \quad H^1(\Gamma, \mT^*) = 0. \] 
Indeed, the zeroth cohomology $H^0(\Gamma, \mT^*)$ is identified with the space of modular forms of weight $3/2$, a one-dimensional space spanned by~$\theta^3$, as can be deduced from \cite{CO76dimensions}. As for $H^1(\Gamma, \mT^*)$, its dimension is equal to the multiplicity of $\mT^{cl}$ in the space of cusp forms on $\Gamma$. The representation $\mT^{cl}$ has highest weight $-1/2$, and the vanishing of   $H^1$ results from the absence of holomorphic cusp forms of weight $1/2$ on $\Gamma$ as in \ref{lemma spaces of mf}. 
We then deduce an interpolation theorem as in \S \ref{introduction}, noting that in addition to the $\delta_{n}$ the distributions $\phi \mapsto \phi'(0)$ (resp. $\phi \mapsto \hat{\phi}'(0)$) are also $e$- (resp. $f$-)invariant.
Arguing as in \S \ref{outline} recovers a non-explicit version of the interpolation theorem of Radchenko-Viazovska for odd Schwartz functions, see \cite[Thm. 7]{RV19}.

\subsubsection{Radial Schwartz functions on $\mathbf{R}^d$} We may, similarly, consider instead the representation $\mathcal{S}_d$ of $\SL_2(\R)$ on radial Schwartz functions on $\mathbf{R}^d$. 
This is, for reasons very similar to that enunciated in \S  \ref{Weil}, 
a lowest weight representation of the double cover of $\SL_2(\R)$, but now of lowest weight $d/2$ generated by $e^{-\pi(x_1^2+...+x_d^2)}$.
We claim that in all cases the corresponding $H^1$ continues to vanish.
Indeed, for $d$ even the complementary representation $W^{cl}$ is finite-dimensional
and does not occur in cusp forms;  for $d$ odd,
occurrences of $W^{cl}$ in cusp forms correspond just as before to 
 holomorphic cusp forms of weight $\frac{4-d}{2}$ for $\Gamma(2)$, 
 and these 
do not exist for any odd $d$. Therefore we find
that the values of $f$  and $\hat{f}$ at radii $\sqrt{n}$
determine $f$,  subject only to a finite-dimensional space of constraints (the dimension
is equal to that of weight $d/2$ holomorphic forms for $\Gamma(2)$).

\subsubsection{Heisenberg uniqueness} \label{section variant}

A result of Hedenmalm and Montes-Rodríguez \cite{HM}
asserts that the map  
\begin{equation} \label{HMmap} L^1(\mathbf{R}) \rightarrow \mbox{sequences}, \quad   h \mapsto \int h(t) e^{\pi i \alpha n t}dt,  \int h(t) e^{\pi i \beta n/t}dt \end{equation}
 is injective if and only if $\alpha \beta \leq 1$.  In their terminology, this yields an example of a ``Heisenberg uniqueness pair.'' 
We thank the referee for bringing this result to our attention. 
Using our techniques, we show
that an abstract interpolation formula -- admittedly, on a eccentric function space -- holds at the transition point $\alpha \beta=1$.

 \begin{theorem}\label{P1 interpolation} Let
 $\mathcal{H}$ be the space of smooth functions on $\R$ with the property that $x^{-2} h(x^{-1})$ extends 
 from $\R-\{0\}$ to a smooth function on $\R$.
 Fix $\alpha, \beta$ with $\alpha \beta=1$ and for $n \in \BZ$ 
write $a_n =  \int h(t) e^{\pi i  \alpha n t}dt$ and $b_n = \int h(t) e^{ \pi i\beta n/t} dt$.
Then the map 
$$ h \mapsto ((a_n), (b_n), h(0), \lim_{x \rightarrow \infty} x^2 h(x))
$$ 
	  defines a linear isomorphism of     $\mathcal{H}$
	   	with  a  codimension $3$ subspace $S$ of\footnote{Note that integration by parts shows that $a_n, b_n$  indeed belong to the space $\fs$ of sequences with rapid decay, introduced in \S \ref{introduction}.} 
  $\mathfrak{s}^2 \oplus \C^2$.
  	\end{theorem}

In this form, this neither implies nor is implied by the results
of \cite{HM}, but it would be interesting to see if our methods can give results
closer to theirs, e.g. by considering different completions of the
underlying representation.

 We obtain 
 Theorem \ref{P1 interpolation} in a similar   
 way to Theorem \ref{interpolation} -- namely,  by applying Theorem \ref{mainthm} for the same $\Gamma$, but with a different coefficient system. 
Note that we can and will assume that $\alpha=\beta=1$ by rescaling. 
We now consider the space   $W$  of smooth $1$-forms on   $\mathbf{P}^1_{\R} = \R \cup \infty$,
which we may think of equivalently as   smooth functions  $\Phi(x,y)$ on $\mathbf{R}^2-\{0\}$ 
satisfying
$$\Phi(\lambda x, \lambda y) = \lambda^{-2} \Phi(x, y).$$
The $1$-form  on $\mathbf{P}^1_{\R}$ associated to $\Phi$ is characterized by the fact that, when pulled back to $\R^2-\{0\}$,
it gives the $1$-form   $\Phi(x,y) (x dy -y dx)$.
 Write
\begin{equation} \label{hPhi}  a_n(\Phi) = \int \Phi(x, 1) e^{\pi i n x} dx, \quad  b_n(\Phi)  = \int \Phi(1, y) e^{\pi i n y}dy .\end{equation}

Write $h(x) = \Phi(x, 1)$. We note that  that $x^{-2} h(1/x) = \Phi(1, x)$, and so extends over~$0$. 
The map  $ \Phi \mapsto h(x) = \Phi(x, 1)$
thus identifies $W$ with the space $\mathcal{H}$
described in the theorem.
We are reduced
then to proving: 
 
\begin{quote}
{\em Claim:} the rule
\begin{equation} \label{Phi interpolation} \Phi \mapsto (a_n, b_n,\Phi(0,1), \Phi(1,0))\end{equation}
defines an isomorphism  of  $W$ with a codimension $3$ subspace of $\mathfrak{s}^2 \oplus \mathbf{C}^2$.
\end{quote}

\noindent \emph{Proof of Claim:}
We apply Theorem \ref{mainthm} to the $\gK$-module $W_K$;
the distribution globalization ``$W^*_{-\infty}$''  that appears in Theorem \ref{mainthm}
is simply the topological dual $W^*$ to $W$. 

To analyze  
  the $e$-invariants on $W^*$, take an arbitrary $e$-invariant distribution $\mathcal{D}$ on $W$.
  The identification $\Phi \mapsto h$ between the space of $-2$-homogeneous $\Phi$
  and $h \in \mH$ 
   contains $C^{\infty}_c(\R)$ in its image;
  thus, we can consider  $\mathcal{D}$ as a distribution on  the real line, i.e.,
  given any $h \in C^{\infty}_c(\R)$, we form the corresponding $\Phi$ and evaluate $\mathcal{D}$ on it. 
The result is a periodic distribution under $x \mapsto x+2$
which must be in the closed subspace spanned by the $a_n$ for $n \in \Z$ -- write this distribution $\sum c(n)a_n$. 
  Then the difference $\mathcal{D} - \sum c(n) a_n$ vanishes on $C^{\infty}_c(\R)$,
and is therefore a linear combination of the Taylor coefficients of $\Phi(1,y)$ at $y=0$; the only 
such distribution that is invariant under $e$ is $\Phi \mapsto \Phi(1,0)$.  It follows that $(W^*)^e$ is spanned topologically
 by the $a_n$ and evaluation at $(1,0)$. Similarly, $(W^*)^f$
 is spanned topologically by the $b_n$ and evaluation at $(0,1)$.

We will now  compute the cohomology of $\Gamma$ on $W^*$. 

 The space $W$ is identified with a reducible principal series of $SL_2(\BR)$
 which is an    extension $  D_2^+ \oplus D_2^- \to W \to \BC $, where $D_2^{\pm}$
 are the holomorphic and antiholomorphic discrete series of weight $2$;
 the map $W \rightarrow \BC$ is the integration over $\mathbf{P}^1_{\R}$. 
   Now Theorem \ref{mainthm} implies that $H^1(\Gamma,(D_2^{\pm})^*)$ vanishes, 
   whereas $H^0(\Gamma, (D_2^{\pm})^*)$ has dimension $2$ in both the $+$ and $-$ case. 
There is therefore an exact sequence 
   $$0 \rightarrow \BC \rightarrow H^0(\Gamma, W^*) \rightarrow \C^4 \rightarrow H^1(\Gamma, \BC) \rightarrow H^1(\Gamma, W^*) \rightarrow 0.$$
      The map $\C^4 \rightarrow H^1(\Gamma, \BC)$ is  surjective
      with two-dimensional kernel,
      for it amounts to the the map from the $4$-dimensional space of (holomorphic and antiholomorphic) Eisenstein series 
      for $\Gamma$ to the two-dimensional cohomology.
      \footnote{ Indeed, this map records
      the obstruction to extending an embedding of $D_2^+ \oplus D_2^-$ into the space
      of automorphic forms, to the larger space $W$. An embedding of $D_2^+ \oplus D_2^-$
      into the space of automorphic forms
      corresponds to a pair $(f, g)$ of a holomorphic and antiholomorphic $1$-form,
      and it extends to $W$ when $f dz +g\overline{dz}$ is an exact differential.} This proves
   that    $H^1(\Gamma, W^*)$ vanishes, whereas $H^0(\Gamma, W^*)$ is $3$-dimensional.  

 We now apply the Mayer-Vietoris sequence \eqref{MWseq}.
 In our current context, it implies that
$(W^*)^e$ and $(W^*)^f$ span all of $W^*$, and their intersection is precisely $3$-dimensional.
This concludes the proof of the claim. 
\qed

 It may be of interest to describe  the three linear constraints that define
 this codimension $3$ subspace. 
 We follow the notations above. 
 The invariants of $\Gamma$ on $W^*$
 have, as basis $A, I, J$
where
$$A(\Phi) = \int_{\mathbf{P}^1} \Phi$$
    $$I(\Phi) = \sum_{(m,n) \neq (0,0)} \Phi(m, n)  - 2 \sum_{2|n} \Phi(m,n)$$
  $$J(\Phi) =\sum_{(m,n) \neq (0,0)} \Phi(m, n)  - 2 \sum_{2|m} \Phi(m,n)$$
where in both cases the sum is conditionally convergent (e.g., one can sum
over large discs of increasing radii). 
Then 
$A$ corresponds to the relation $a_0=b_0$, 
whereas both $I$ and $J$ 
give rise to a relation by   expanding the stated intertwiner in two different ways. 
For example, we compute $I(\Phi)$ in two ways, firstly by summing first over $n$:
 \begin{eqnarray*}   \sum_{m \neq 0}  m^{-2}\sum_{n} (  \Phi(1, \frac{n}{m})  -2  \Phi(1, \frac{2n}{m}))+ 2 ( \sum n^{-2}  -2 \sum (2n)^{-2}) \Phi(0,1)
\\\stackrel{\textrm{P.S.}}{=}  
 -\sum_{m \neq 0, t \in \Z} |m|^{-1} b_{m (2t+1)} + \frac{\pi^2}{6} \Phi(0,1)
\end{eqnarray*}
where P.S. stands for Poisson summation, 
and secondly by summing  first over $m$:
\begin{eqnarray*}   \sum_{n \neq 0} \left(  n^{-2}    \sum_{m} \Phi(\frac{m}{n}, 1)   - \sum_m 2  (2n)^{-2} \Phi(\frac{m}{2n},1) \right) 2 ( \sum m^{-2}  -  2 \sum (m)^{-2}) \Phi(1,0)
\\ \stackrel{\textrm{P.S.}}{=}   \sum |n|^{-1}  a_{n(4t+2)}
 -  \frac{\pi^2}{3} \Phi(1,0). 
\end{eqnarray*}
 Thus we find that the image of $W$ is cut out by the three relations $a_0=b_0$, 
 $$
  \frac{\pi^2}{3} \Phi(1,0)+ \frac{\pi^2}{6} \Phi(0,1) =
 \sum_{m \neq 0, t \in \Z} |m|^{-1} b_{m (2t+1)} +  
  \sum_{n \neq 0, t \in \Z} |n|^{-1}  a_{n(4t+2)},
 $$
and dually
 $$
  \frac{\pi^2}{6} \Phi(1,0)+ \frac{\pi^2}{3} \Phi(0,1) =
 \sum_{m \neq 0, t \in \Z} |m|^{-1}a_{m (2t+1)} +  
  \sum_{n \neq 0, t \in \Z} |n|^{-1}  b_{n(4t+2)}.
 $$

 \bibliography{interpolation}{}
\bibliographystyle{alpha}

\end{document}